\documentclass[twoside]{aiml22}

\usepackage{aiml22macro}

\usepackage{amsmath}
\usepackage{amssymb}

\usepackage{stmaryrd}
\usepackage{mathrsfs}

\usepackage{tikz}
\usepackage{pgf} 
\usetikzlibrary{patterns,automata,arrows,shapes,snakes,topaths,trees,backgrounds,positioning,through,calc,arrows.meta}

\definecolor{darkgreen}{rgb}{0.0, 0.7, 0.0}

\newcommand{\comp}{\vartriangleleft}
\newcommand{\compflip}{\vartriangleright}
\newcommand{\lat}{\mathfrak{L}}

\newtheorem{question}[thm]{Question}
\newtheorem{convention}[thm]{Convention}

\newtheorem{manualtheoreminner}{Theorem}
\newenvironment{manualtheorem}[1]{%
  \renewcommand\themanualtheoreminner{#1}%
  \manualtheoreminner
}{\endmanualtheoreminner}

\newtheorem{manualpropinner}{Proposition}
\newenvironment{manualprop}[1]{%
  \renewcommand\themanualpropinner{#1}%
  \manualpropinner
}{\endmanualtheoreminner}

\def\lastname{Holliday}

\usepackage{comment}

\begin{document}

\begin{frontmatter}
  \title{Compatibility and accessibility:\\lattice representations for semantics\\of non-classical and modal logics}

  \author{Wesley H. Holliday}
  \address{University of California, Berkeley}
  
    \subtitle{{\footnotesize Published in \textit{Advances in Modal Logic}, Vol.~14, 2022, 507-529.}}
 
 \begin{abstract} In this paper, we study three representations of lattices by means of a set with a binary relation of compatibility in the tradition of Plo\v{s}\v{c}ica. The standard representations of complete ortholattices and complete perfect Heyting algebras drop out as special cases of the first representation, while the second covers arbitrary complete lattices, as well as complete lattices equipped with a negation we call a protocomplementation. The third topological representation is a variant of that of Craig, Havier, and Priestley. We then extend each of the three representations to lattices with a multiplicative unary modality; the representing structures,  like so-called graph-based frames, add a second relation of accessibility interacting with compatibility. The three representations generalize possibility semantics for classical modal logics to non-classical modal logics, motivated by a recent application of modal orthologic to natural language semantics.\end{abstract}
   
 \end{frontmatter}
 
 \makeatletter
\def\@title{Compatibility and accessibility}
\makeatother 
 
\section{Introduction}\label{Intro}

Semantics for non-classical and modal logics may be seen as arising from more basic algebraic representation theorems. For example, traditional semantics for intuitionistic logic, orthologic, and classical modal logic may be understood in terms of the following well-known representations:
\begin{itemize}
\item Any Heyting algebra $H$ embeds into the lattice of downsets of a poset, and the embedding is an isomorphism if $H$ is complete and its completely join-irreducible elements are join-dense (see, e.g., \cite{Esakia2019}, \cite[Prop.~1.1]{Davey1979}).
\item Any ortholattice $L$ embeds into the lattice of $\bot$-closed sets of an orthoframe $(X,\bot)$ equipped  with an orthocomplementation $\neg$ induced by  the relation~$\bot$, and the embedding is an isomorphism if $L$ is complete (see \cite{Goldblatt1974} for the orthoframe description and \cite[\S\S~32-4]{Birkhoff1940}, \cite{Dishkant1972} for other descriptions).
\item Any Boolean algebra $B$ equipped with a multiplicative unary operation $\Box$ embeds into the powerset of a set $W$ equipped with an operation $\Box_R$ induced by a binary relation $R$ on $W$, and the embedding is an isomorphism if $B$ is complete and atomic and $\Box$ completely multiplicative (see \cite{Tarski1935,Jonsson1952a,Thomason1975}).
\end{itemize}
In each case, adding topology to the relevant relational structures allows one to characterize topologically the image of the relevant embedding \cite{Esakia2019,Goldblatt1975a,Bimbo2007,Goldblatt1974a,Sambin1988}.

Here we study representations that subsume and go beyond all of those mentioned above. In \S~\ref{FramesToLattices}, we explain how to go from a set together with a binary relation of ``compatibility'' to a complete lattice. In \S~\ref{LatticesToFrames}, we study three ways of going back: one economical representation of certain complete lattices, including but not limited to Heyting and ortholattice cases; one less economical but fully general representation of complete lattices, including complete lattices equipped with a type of negation that we call a protocomplementation; and one representation of arbitrary lattices. In \S~\ref{ModalCase}, we extend the three representations to lattices with a multiplicative unary modality $\Box$, by adding a second relation of accessibility interacting with compatibility. We conclude in \S~\ref{Conclusion}.

After writing this paper, I discovered Plo\v{s}\v{c}ica's \cite{Ploscica1993} representation of bounded lattices using certain compatibility frames\footnote{Also see \cite{Craig2015} for \textit{TiRS graphs}, which are compatibility frames with extra properties, which we do not require here (as is crucial for a number of our results and for Conjecture \ref{Conj}).} as in Definition \ref{CompFrames} with a topology, as well as Craig et al.'s~\cite{Craig2013} modification of Plo\^s\^cica's approach. In \S~\ref{SubRepSec}, we briefly cover a variant of this representation with a different topology. In addition, a referee informed me that the addition of modal accessibility interacting with compatibility  (Definition \ref{CAframes}) appears in the \textit{graph-based frames} of Conradie et al.~\cite{Conradie2019}. We will return to this connection in \S~\ref{ModalCase}.

A Jupyter notebook with code to verify examples and investigate conjectures and questions is available at \href{https://github.com/wesholliday/compat-frames}{github.com/wesholliday/compat-frames}.

\section{From compatibility frames to lattices}\label{FramesToLattices}

\subsection{Basic concepts}

Our starting point is a certain way of going from a set with a binary relation to a complete lattice. For a comparison with other ways of realizing complete lattices using doubly ordered structures and polarities, see \cite{Holliday2021}.

\begin{definition}\label{CompFrames} A \textit{relational frame} is a pair $\mathcal{F}=(X,\comp)$ where $X$ is a nonempty set and $\comp$ is a binary relation on $X$. A \textit{compatibility frame} is a relational frame in which $\comp$ is reflexive.
\end{definition}
\noindent We read  $x\comp y$ as ``$x$ is compatible with $y$,'' also written $y\compflip x$.\footnote{In \cite{Holliday2021}, we wrote $x\between y$ and $y\between^{-1}x$ instead of $x\comp y$ and $y\compflip x$, respectively.} 

\begin{convention}\label{DiagramConv} \textnormal{In diagrams, such as Fig.~\ref{FirstFig}, an arrow with a triangle arrowhead from $y$ to $x$ indicates $y\compflip x$. Thus, we draw the directed graph $(X,\compflip)$ to represent the compatibility frame $(X,\comp)$. Reflexive loops are not shown.}
\end{convention}

Recall that a unary operation on a lattice  is a \textit{closure operator} if $c$ is inflationary ($x\leq c(x)$), idempotent ($c(c(x))=c(x)$), and monotone ($x\leq y$ implies $c(x)\leq c(y)$). We will use the compatibility relation $\comp$ to define a closure operator on $\wp(X)$, whose fixpoints give us a complete lattice as in the following classic result (see, e.g., \cite[Thm.~5.2]{Burris1981}).

\begin{proposition}\label{ClosureLattice} Let $X$ be a nonempty set and $c$ a closure operator on $\wp(X)$. Then the fixpoints of $c$, i.e., those $A\subseteq X$ with $c(A)=A$, ordered by $\subseteq$ form a complete lattice with
\[\underset{i\in I}{\bigwedge}{A_i} =\underset{i\in I}{\bigcap}{A_i}\mbox{ and } \underset{i\in I}{\bigvee}{A_i} =c(\underset{i\in I}{\bigcup}{A_i}).\]
\end{proposition}

\begin{definition} Given a relational frame $(X,\comp)$, define $c_\comp: \wp(X)\to\wp(X)$ by
\[c_\comp(A)=\{x\in X\mid \forall x'\comp x\; \exists x''\compflip x':\, x''\in A\}.\] 
\end{definition}
\noindent Thus, $x$ is in $c_\comp(A)$ iff every state compatible with $x$ is compatible with some state in $A$. Given a compatibility frame, we are interested in the $c_\comp$-fixpoints, i.e., those $A\subseteq X$ such that $c_\comp(A)=A$. Looking at a diagram of a compatibility frame, one can check that $c_\comp(A)=A$ by checking that the following holds: 
\begin{itemize}
\item from any $x\in X\setminus A$, you can step forward along an arrow to a state $x'$ that cannot step backward along an arrow into $A$. 
\end{itemize}
Informally, ``from $x$ you can see a state that cannot be seen from $A$.'' 

\begin{example} Consider the cycle on three elements on the left of Fig.~\ref{FirstFig}, regarded as a compatibility frame according to Convention \ref{DiagramConv}: $\{y\}$ is a {$c_\comp$-fixpoint} because $z$ and $x$ can both see $x$, which cannot be seen from $\{y\}$. Yet  $\{y,z\}$ is not a $c_\comp$-fixpoint, because  $x$ cannot see a state that cannot be seen from $\{y,z\}$, since both $x$ and $y$ can be seen from $\{y,z\}$.

We get the reverse verdicts on $\{y\}$ and $\{y,z\}$ in the acyclic (ignoring loops) but non-transitive frame on the right of Fig.~\ref{FirstFig}: $\{y\}$ is \textit{not} a $c_\comp$-fixpoint, because now $z$ cannot see a state that cannot be seen from $\{y\}$; but $\{y,z\}$ \textit{is} a $c_\comp$-fixpoint, because $x$ can see a state, namely $x$, that cannot be seen from $\{y,z\}$.\end{example}

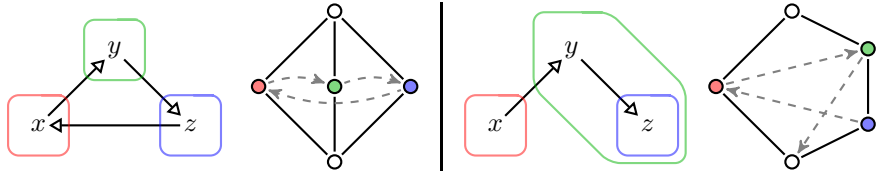
\begin{figure}[h]
\begin{center}
\begin{minipage}{1.1in}
\begin{center}
\begin{tikzpicture}[->,>=stealth',shorten >=1pt,shorten <=1pt, auto,node
distance=2cm,thick,every loop/.style={<-,shorten <=1pt}]
\tikzstyle{every state}=[fill=gray!20,draw=none,text=black]
\node[label=center:$x$,inner sep=0pt,minimum size=.175cm] at (0,0) (D) {}; 
\node[label=center:$y$,inner sep=0pt,minimum size=.175cm] at (1,1) (F) {}; 
\node[label=center:$z$,inner sep=0pt,minimum size=.175cm] at (2,0) (H) {}; 

\path[-{Triangle[open]},draw,thick] (D) to node {{}}  (F);
\path[-{Triangle[open]},draw,thick] (F) to node {{}}  (H);
\path[-{Triangle[open]},draw,thick] (H) to node{{}}  (D);

\path[-, draw=red, opacity=0.5, thick, rounded corners]  (0, .4) -- (.4, .4) -- (.4, -.4) -- (-.4, -.4) -- (-.4, .4) -- (.4, .4) -- (0, .4); 

\path[-, draw=darkgreen, opacity=0.5, thick, rounded corners] (1, 1.4) -- (1.4, 1.4) -- (1.4, .6) -- (.6, .6) -- (.6, 1.4) -- (1.4, 1.4) -- (1, 1.4); 

\path[-, draw=blue, opacity=0.5, thick, rounded corners] (2, .4) -- (2.4, .4) -- (2.4, -.4) -- (1.6, -.4) -- (1.6, .4) -- (2.4, .4) -- (2, .4); 

\end{tikzpicture}
\end{center}
\end{minipage}
\begin{minipage}{1.1in}
\begin{center} 
\begin{tikzpicture}[->,>=stealth',shorten >=1pt,shorten <=1pt, auto,node
distance=2cm,thick,every loop/.style={<-,shorten <=1pt}]
\tikzstyle{every state}=[fill=gray!20,draw=none,text=black]
\node[circle,draw=black!100, label=right:$$,inner sep=0pt,minimum size=.175cm] (1) at (0,0) {{}};
\node[circle,fill=red!50,draw=black!100, label=left:$$,inner sep=0pt,minimum size=.175cm] (x) at (-1,-1) {{}};
\node[circle,fill=darkgreen!50, draw=black!100, label=right:$$,inner sep=0pt,minimum size=.175cm] (y) at (0,-1) {{}};
\node[circle,fill=blue!50,draw=black!100, label=right:$$,inner sep=0pt,minimum size=.175cm] (z) at (1,-1) {{}};
\node[circle,draw=black!100, label=right:$$,inner sep=0pt,minimum size=.175cm] (0) at (0,-2) {{}};

\path (1) edge[-] node {{}} (y);
\path (1) edge[-] node {{}} (x);
\path (1) edge[-] node {{}} (z);
\path (x) edge[-] node {{}} (0);
\path (y) edge[-] node {{}} (0);
\path (z) edge[-] node {{}} (0);

\path (x) edge[<-,dashed,bend right=20, gray] node {{}} (z);
\path (z) edge[<-,dashed,bend right=20, gray] node {{}} (y);
\path (y) edge[<-,dashed,bend right=20, gray] node {{}} (x);

\end{tikzpicture}
\end{center}
\end{minipage}\vrule\;\;\;\begin{minipage}{1.15in}
\begin{center}
\begin{tikzpicture}[->,>=stealth',shorten >=1pt,shorten <=1pt, auto,node
distance=2cm,thick,every loop/.style={<-,shorten <=1pt}]
\tikzstyle{every state}=[fill=gray!20,draw=none,text=black]
\node[label=center:$x$,inner sep=0pt,minimum size=.175cm] at (0,0) (D) {}; 
\node[label=center:$y$,inner sep=0pt,minimum size=.175cm] at (1,1) (F) {}; 
\node[label=center:$z$,inner sep=0pt,minimum size=.175cm] at (2,0) (H) {}; 

\path[-{Triangle[open]},draw,thick] (D) to node {{}}  (F);
\path[-{Triangle[open]},draw,thick] (F) to node {{}}  (H);

\path[-, draw=red, opacity=0.5, thick, rounded corners]  (0, .4) -- (.4, .4) -- (.4, -.4) -- (-.4, -.4) -- (-.4, .4) -- (.4, .4) -- (0, .4); 

\path[-, draw=darkgreen, opacity=0.5, thick, rounded corners] (1.1, 1.5) -- (1.5, 1.5) -- (2.5, .5) -- (2.5, -.5) -- (1.5, -.5) -- (.5, .5)  -- (.5, 1.5) -- (1.5, 1.5) -- (1.1, 1.5); 

\path[-, draw=blue, opacity=0.5, thick, rounded corners] (2, .4) -- (2.4, .4) -- (2.4, -.4) -- (1.6, -.4) -- (1.6, .4) -- (2.4, .4) -- (2, .4); 

\end{tikzpicture}
\end{center}
\end{minipage}
\begin{minipage}{1.1in}
\begin{center}
\begin{tikzpicture}[->,>=stealth',shorten >=1pt,shorten <=1pt, auto,node
distance=2cm,thick,every loop/.style={<-,shorten <=1pt}]
\tikzstyle{every state}=[fill=gray!20,draw=none,text=black]

\node[circle,draw=black!100, label=right:$$,inner sep=0pt,minimum size=.175cm] (n1) at (0,0) {{}};
\node[circle,fill=darkgreen!50,draw=black!100, label=left:$$,inner sep=0pt,minimum size=.175cm] (nx) at (1,-.5) {{}};
\node[circle,fill=blue!50,draw=black!100, label=right:$$,inner sep=0pt,minimum size=.175cm] (ny) at (1,-1.5) {{}};
\node[circle,fill=red!50,draw=black!100, label=right:$$,inner sep=0pt,minimum size=.175cm] (nz) at (-1,-1) {{}};
\node[circle,draw=black!100, label=right:$$,inner sep=0pt,minimum size=.175cm] (n0) at (0,-2) {{}};
\path (nx) edge[-] node {{}} (n1);
\path (nx) edge[-] node {{}} (ny);
\path (n1) edge[-] node {{}} (nz);
\path (ny) edge[-] node {{}} (n0);
\path (nz) edge[-] node {{}} (n0);

\path (ny) edge[->,dashed,gray] node {{}} (nz);

\path (nx) edge[->,dashed,gray] node {{}} (n0);

\path (nz) edge[->,dashed,gray] node {{}} (nx);

\end{tikzpicture}
\end{center}
\end{minipage}
\end{center}
\caption{Two compatibility frames drawn according to Convention \ref{DiagramConv} with their $c_\comp$-fixpoints (except $X$ and $\varnothing$) outlined, followed by their associated lattices.}\label{FirstFig}
\end{figure}

\begin{proposition}\label{IsClosure} For any relational frame, $c_\comp$ is a closure operator on $\wp(X)$.
\end{proposition}
 \begin{proof} That  $Y\subseteq c_\comp (Y)$  and that $Y\subseteq Z$ implies $c_\comp(Y)\subseteq c_\comp (Z)$ are obvious. To see  $c_\comp(c_\comp(Y))\subseteq c_\comp(Y)$, suppose $x\in c_\comp(c_\comp(Y))$ and $x'\comp x$. Hence there is an  $x''\compflip x'$ with $x''\in c_\comp(Y)$. This implies there is an $x'''\compflip x'$ with $x'''\in Y$.  Thus, for any $x'\comp x$ there is an $x'''\compflip x'$ with $x'''\in Y$. Therefore, $x\in c_\comp(Y)$.
\end{proof}

Given Propositions \ref{ClosureLattice} and \ref{IsClosure}, we have the following immediate corollary.

\begin{corollary}\label{FrameToLate} For any relational frame $(X,\comp)$, the $c_\comp$-fixpoints  ordered by $\subseteq$ form a complete lattice $\lat(X,\comp)$ with meet and join as in Proposition \ref{ClosureLattice}.
\end{corollary}

\begin{example}We see in Fig.~\ref{FirstFig} that the $\mathbf{M}_3$ lattice (ignoring the dashed arrows for now) arises from the cycle on three elements, while the $\mathbf{N}_5$ lattice arises from the acyclic but non-transitive frame on three elements.
\end{example}

We can relate Corollary \ref{FrameToLate} to possible world semantics for classical and intuitionistic logic as follows. Let $=_W$ be the identity relation on the set $W$.

\begin{proposition}\label{PWS}$\,$
\begin{enumerate}
\item\label{PWS1} Given any set $W$, the pair $(W,=_W)$ is a compatibility frame, and ${\lat(W,=_W)}$ is the Boolean algebra of all subsets of $W$.
\item\label{PWS2} Given any preorder $\leq$ on a set $P$, the pair $(P,\leq)$ is a compatibility frame, and $\lat(P,\leq)$ is the Heyting algebra of all downsets of $(P,\leq)$.
\end{enumerate}
\end{proposition}
Part (ii) appears in \cite[Prop.~4.1.1]{Conradie2020} but we include a proof for convenience. 

\begin{proof} Part (\ref{PWS1}) is obvious. For (\ref{PWS2}), let $A$ be a downset and $x\in P\setminus A$. Setting $x'=x$, we have $x'\comp x$, and for all $x''\compflip x'$, i.e., all $x''\geq x'$, $x''\not\in A$, since $A$ is a downset. Thus, $c_\comp(A)=A$. Conversely, suppose $c_\comp(A)=A$, $x\in A$, and $y\leq x$. Let $y'\comp y$, so $y'\leq y$ and hence $y'\leq x$. Then setting $y''=x$, we have $y'\comp y''\in A$. Since  $c_\comp(A)=A$, it follows that $y\in A$. Thus, $A$ is a downset.\end{proof}

\begin{example} To illustrate part (\ref{PWS2}), if we add to the non-transitive frame in Fig.~\ref{FirstFig} the arrow from $x$ to $z$ required by transitivity, then instead of realizing  $\mathbf{N}_5$, we realize the four-element chain in Fig.~\ref{HeytingFig}, which is a Heyting algebra.
\end{example}

\begin{figure}[h]
\begin{center}
\begin{minipage}{1.5in}
\begin{center}
\begin{tikzpicture}[->,>=stealth',shorten >=1pt,shorten <=1pt, auto,node
distance=2cm,thick,every loop/.style={<-,shorten <=1pt}]
\tikzstyle{every state}=[fill=gray!20,draw=none,text=black]
\node[label=center:$x$,inner sep=0pt,minimum size=.175cm] at (0,0) (D) {}; 
\node[label=center:$y$,inner sep=0pt,minimum size=.175cm] at (1,1) (F) {}; 
\node[label=center:$z$,inner sep=0pt,minimum size=.175cm] at (2,0) (H) {}; 

\path[-{Triangle[open]},draw,thick] (D) to node {{}}  (F);
\path[-{Triangle[open]},draw,thick] (F) to node {{}}  (H);
\path[-{Triangle[open]},draw,thick] (D) to node {{}}  (H);

\path[-, draw=darkgreen, opacity=0.5, thick, rounded corners] (1.1, 1.5) -- (1.5, 1.5) -- (2.5, .5) -- (2.5, -.5) -- (1.5, -.5) -- (.5, .5)  -- (.5, 1.5) -- (1.5, 1.5) -- (1.1, 1.5); 

\path[-, draw=blue, opacity=0.5, thick, rounded corners] (2, .4) -- (2.4, .4) -- (2.4, -.4) -- (1.6, -.4) -- (1.6, .4) -- (2.4, .4) -- (2, .4); 

\end{tikzpicture}
\end{center}
\end{minipage}
\begin{minipage}{1.5in}
\begin{center}
\begin{tikzpicture}[->,>=stealth',shorten >=1pt,shorten <=1pt, auto,node
distance=2cm,thick,every loop/.style={<-,shorten <=1pt}]
\tikzstyle{every state}=[fill=gray!20,draw=none,text=black]

\node[circle,draw=black!100, label=right:$$,inner sep=0pt,minimum size=.175cm] (n1) at (0,0) {{}};
\node[circle,fill=darkgreen!50,draw=black!100, label=left:$$,inner sep=0pt,minimum size=.175cm] (nx) at (0,-.75) {{}};
\node[circle,fill=blue!50,draw=black!100, label=right:$$,inner sep=0pt,minimum size=.175cm] (ny) at (0,-1.5) {{}};
\node[circle,draw=black!100, label=right:$$,inner sep=0pt,minimum size=.175cm] (n0) at (0,-2.25) {{}};
\path (nx) edge[-] node {{}} (n1);
\path (nx) edge[-] node {{}} (ny);
\path (ny) edge[-] node {{}} (n0);

\path (nx) edge[->,bend left=25,dashed,gray] node {{}} (n0);
\path (ny) edge[->,bend right=25,dashed,gray] node {{}} (n0);

\end{tikzpicture}
\end{center}
\end{minipage}

\end{center}
\caption{A compatibility frame realizing a Heyting algebra.}\label{HeytingFig}
\end{figure}
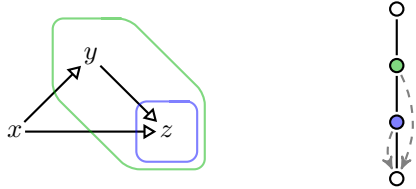

Some appealing aspects of working with downsets of a preorder also apply to our fixpoints; let $^\partial$, $\oplus$, $\overline{\oplus}$, $\dot{\cup}$, and $\times$ be the dual, linear sum, vertical sum, disjoint union, and product  operations  \cite[\S~1.31, \S~1.24, Ex.~1.18, Prop.~1.32]{Davey2002}.

\begin{proposition}\label{Sums} For any relational frames $\mathcal{F}$ and $\mathcal{G}$, \textnormal{(i)} $\lat(\mathcal{F}^\partial)\cong\lat(\mathcal{F})^\partial$; \textnormal{(ii)}  $\lat(\mathcal{F}\oplus \mathcal{G})\cong\lat(\mathcal{F})\,\overline{\oplus}\, \lat(\mathcal{G})$; and \textnormal{(iii)}  $\lat(\mathcal{F}\,\dot{\cup}\, \mathcal{G})\cong\lat(\mathcal{F})\times \lat(\mathcal{G})$.
\end{proposition}
\begin{proof} Let $\mathcal{F}=(X,\comp)$, so $\mathcal{F}^\partial=(X,\compflip)$. For an isomorphism $\varphi$ from $\lat(\mathcal{F}^\partial)$ to $\lat(\mathcal{F})^\partial$, let   $\varphi(A)=\{x\in X\mid \forall y\comp x\; y\not\in A \}$. From $\lat(\mathcal{F})\overline{\oplus} \lat(\mathcal{G})$ to $\lat(\mathcal{F}\oplus \mathcal{G})$,  $\varphi(A)=A$ if $A\in \lat(\mathcal{F})$, and $\varphi(A)=A\cup X$ if $A\in \lat(\mathcal{G})$.~Part (iii) is also easy.\end{proof}

We can also relate our approach to that of realizing complete Boolean algebras as in forcing \cite{Takeuti1973} or possibility semantics \cite{Holliday2021b} for classical logic as follows.

\begin{proposition}\label{RegOpen} Given a preordered set $(P,\leq)$, define $\comp$ on $P$ by: $x\comp y$ if $\exists z\in P$:~$z\leq x$ and $z\leq y$. Then $(P,\comp)$ is a compatibility frame, and $\lat(P,\comp)$ is the Boolean algebra $\mathcal{RO}(P,\leq)$ of all regular open downsets of $(P,\leq)$.
\end{proposition}
\begin{proof} Observe that any $c_\comp$-fixpoint is a $\leq$-downset, and for any $\leq$-downset~$A$,  (i) $\forall x'\leq x\,\exists x''\leq x':\, x''\in A$ and (ii)~$\forall x'\comp x\,\exists x''\compflip x':\, x''\in A$ are equivalent. It follows that $A$ is a regular open downset iff $A$ is a $c_\comp$-fixpoint.
\end{proof}

Now $(X,\comp)$ gives us not only $\mathfrak{L}(X,\comp)$ but also an operation $\neg_\comp$ on $\mathfrak{L}(X,\comp)$.

\begin{proposition}\label{NegProp} For any relational frame $(X,\comp)$ and $A\subseteq X$, the set
$\neg_\comp A = \{x\in X\mid \forall y\comp x\;\; y\not\in A \}$
is a $c_\comp$-fixpoint.
\end{proposition}
\begin{proof} If $x\in X\setminus \neg_\comp A$, then $\exists x'\comp x$ with $x'\in A$, so $\forall x''\compflip x'$, $x''\not\in \neg_\comp A$.
\end{proof}

To characterize the $\neg_\comp$ operation, let us recall some terminology.

\begin{definition} Let $L$ be a bounded lattice and $a\in L$. An $x\in L$ is a \textit{semicomplement} of $a$ if $a\wedge x=0$, a \textit{complement} of $a$ if $a\wedge x=0$ and $a\vee x=1$, and a \textit{pseudocomplement} of $a$ if $x$ is the maximum in $L$ of $\{y\in X\mid a\wedge y=0\}$. 

A unary operation $\neg$ on $L$ is a \textit{semicomplementation} (resp.~\textit{complementation}, \textit{pseudocomplementation}) if for all $a\in L$, $\neg a$ is a semicomplement (resp.~complement, pseudocomplement) of $a$. It is \textit{antitone} if for all $a,b\in L$, $a\leq b$ implies $\neg b\leq \neg a$, \textit{involutive} if $\neg\neg a=a$ for all $a\in L$, and  \textit{anti-inflationary} if $a\not\leq \neg a$ for all nonzero $a\in L$.  An \textit{ortholattice} is a bounded lattice equipped with an involutive antitone complementation, called an \textit{orthocomplementation}. A \textit{p-algebra} is a bounded lattice equipped with a pseudocomplementation. 

Finally, for a non-standard piece of terminology, we say $\neg$ is a \textit{protocomplementation} if $\neg$ is an antitone semicomplementation such that $\neg 0=1$.

\end{definition}

An antitone $\neg$ is anti-inflationary iff it is a semicomplementation. Also recall that the operation in a Heyting algebra $H$ defined by $\neg a=a\to 0$ is a pseudocomplementation, so $H$ may also be regarded as a p-algebra; and the complementation in a Boolean algebra $B$ is an orthocomplementation, so $B$ is an ortholattice. As for the operation $\neg_\comp$, the following is easy to check.

\begin{proposition} For any compatibility \textnormal{(}resp.~relational\textnormal{)} frame $(X,\comp)$,   $\neg_\comp$  is a protocomplementation \textnormal{(}resp.~antitone and such that $\neg 1=0$\textnormal{)} on  $\lat(X,\comp)$.
\end{proposition}

\noindent In our diagrams of lattices arising from compatibility frames, the dashed arrows represent the operation $\neg_\comp$. We omit arrows representing  $\neg 0=1$ and $\neg 1=0$.

\subsection{Frames for ortholattices}

If we assume that $\comp$ is \textit{symmetric}, then we get a standard representation (as in \cite{Goldblatt1974} via ``proximity frames'') of ortholattices.  The proof is straightforward.

\begin{proposition}\label{IsOrtho} For any compatibility frame $(X,\comp)$, if $\comp$ is symmetric, then $\neg_\comp$ is an orthocomplementation on $\lat(X,\comp)$.
\end{proposition}

\begin{example} Fig.~\ref{OrthoFig} shows two symmetric compatibility frames  and their associated ortholattices, $\mathbf{MO}_2$ or $\mathbf{M}_4$  (left) and the Benzene ring $\mathbf{O}_6$ (right).
\end{example}

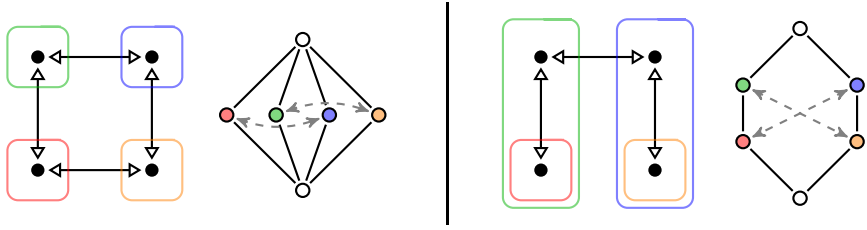
\begin{figure}[h]

\begin{center}
\begin{minipage}{1in}
\begin{center}
\begin{tikzpicture}[->,>=stealth',shorten >=1pt,shorten <=1pt, auto,node
distance=2cm,thick,every loop/.style={<-,shorten <=1pt}]
\tikzstyle{every state}=[fill=gray!20,draw=none,text=black]
\node[circle, fill=black!100,inner sep=0pt,minimum size=.175cm] at (0,0) (A) {}; 
\node[circle, fill=black!100,inner sep=0pt,minimum size=.175cm] at (0,1.5) (B) {}; 
\node[circle, fill=black!100,inner sep=0pt,minimum size=.175cm] at (1.5,1.5) (C) {}; 
\node[circle, fill=black!100,inner sep=0pt,minimum size=.175cm] at (1.5,0) (D) {}; 

\path[{Triangle[open]}-{Triangle[open]},draw,thick] (A) to node {{}}  (B);
\path[{Triangle[open]}-{Triangle[open]},draw,thick] (B) to node {{}}  (C);
\path[{Triangle[open]}-{Triangle[open]},draw,thick] (C) to node {{}}  (D);
\path[{Triangle[open]}-{Triangle[open]},draw,thick] (D) to node {{}}  (A);

\path[-, draw=red, opacity=0.5, thick, rounded corners]  (0, .4) -- (.4, .4) -- (.4, -.4) -- (-.4, -.4) -- (-.4, .4) -- (.4, .4) -- (0, .4); 

\path[-, draw=darkgreen, opacity=0.5, thick, rounded corners]  (0, 1.9) -- (.4, 1.9) -- (.4, 1.1) -- (-.4, 1.1) -- (-.4, 1.9) -- (.4, 1.9) -- (0, 1.9); 

\path[-, draw=blue, opacity=0.5, thick, rounded corners]  (1.5, 1.9) -- (1.9, 1.9) -- (1.9, 1.1) -- (1.1, 1.1) -- (1.1, 1.9) -- (1.9, 1.9) -- (1.5, 1.9); 

\path[-, draw=orange, opacity=0.5, thick, rounded corners]  (1.5, .4) -- (1.9, .4) -- (1.9, -.4) -- (1.1, -.4) -- (1.1, .4) -- (1.9, .4) -- (1.5, .4); 

\end{tikzpicture}
\end{center}
\end{minipage}
\begin{minipage}{1in}
\begin{center}
\begin{tikzpicture}[->,>=stealth',shorten >=1pt,shorten <=1pt, auto,node
distance=2cm,thick,every loop/.style={<-,shorten <=1pt}]
\tikzstyle{every state}=[fill=gray!20,draw=none,text=black]
\node[circle,draw=black!100, label=right:$$,inner sep=0pt,minimum size=.175cm] (1) at (0,0) {{}};
\node[circle,fill=red!50,draw=black!100, label=left:$$,inner sep=0pt,minimum size=.175cm] (x) at (-1,-1) {{}};
\node[circle,fill=darkgreen!50, draw=black!100, label=right:$$,inner sep=0pt,minimum size=.175cm] (y) at (-.35,-1) {{}};
\node[circle,fill=blue!50, draw=black!100, label=right:$$,inner sep=0pt,minimum size=.175cm] (w) at (.35,-1) {{}};
\node[circle,fill=orange!50,draw=black!100, label=right:$$,inner sep=0pt,minimum size=.175cm] (z) at (1,-1) {{}};
\node[circle,draw=black!100, label=right:$$,inner sep=0pt,minimum size=.175cm] (0) at (0,-2) {{}};

\path (1) edge[-] node {{}} (y);
\path (1) edge[-] node {{}} (x);
\path (1) edge[-] node {{}} (z);
\path (1) edge[-] node {{}} (w);
\path (x) edge[-] node {{}} (0);
\path (y) edge[-] node {{}} (0);
\path (z) edge[-] node {{}} (0);
\path (w) edge[-] node {{}} (0);

\path (x) edge[<->,dashed,bend right=20,gray] node {{}} (w);
\path (y) edge[<->,dashed,bend left=20,gray] node {{}} (z);

\end{tikzpicture}
\end{center}
\end{minipage}\qquad\vrule\qquad\begin{minipage}{1in}
\begin{center}
\begin{tikzpicture}[->,>=stealth',shorten >=1pt,shorten <=1pt, auto,node
distance=2cm,thick,every loop/.style={<-,shorten <=1pt}]
\tikzstyle{every state}=[fill=gray!20,draw=none,text=black]
\node[circle, fill=black!100,inner sep=0pt,minimum size=.175cm] at (0,0) (A) {}; 
\node[circle, fill=black!100,inner sep=0pt,minimum size=.175cm] at (0,1.5) (B) {}; 
\node[circle, fill=black!100,inner sep=0pt,minimum size=.175cm] at (1.5,1.5) (C) {}; 
\node[circle, fill=black!100,inner sep=0pt,minimum size=.175cm] at (1.5,0) (D) {}; 

\path[{Triangle[open]}-{Triangle[open]},draw,thick] (A) to node {{}}  (B);
\path[{Triangle[open]}-{Triangle[open]},draw,thick] (B) to node {{}}  (C);
\path[{Triangle[open]}-{Triangle[open]},draw,thick] (C) to node {{}}  (D);

\path[-, draw=red, opacity=0.5, thick, rounded corners]  (0, .4) -- (.4, .4) -- (.4, -.4) -- (-.4, -.4) -- (-.4, .4) -- (.4, .4) -- (0, .4); 

\path[-, draw=darkgreen, opacity=0.5, thick, rounded corners]  (0, 2) -- (.5, 2) -- (.5, -.5) -- (-.5, -.5) -- (-.5, 2) -- (.5, 2) -- (0, 2); 

\path[-, draw=blue, opacity=0.5, thick, rounded corners]  (1.5, 2) -- (2, 2) -- (2, -.5) -- (1, -.5) -- (1, 2) -- (2, 2) -- (1.5, 2); 

\path[-, draw=orange, opacity=0.5, thick, rounded corners]  (1.5, .4) -- (1.9, .4) -- (1.9, -.4) -- (1.1, -.4) -- (1.1, .4) -- (1.9, .4) -- (1.5, .4); 

\end{tikzpicture}
\end{center}
\end{minipage}
\begin{minipage}{1in}
\begin{center}
\begin{tikzpicture}[->,>=stealth',shorten >=1pt,shorten <=1pt, auto,node
distance=2cm,thick,every loop/.style={<-,shorten <=1pt}]
\tikzstyle{every state}=[fill=gray!20,draw=none,text=black]
\node[circle,draw=black!100, label=below:,inner sep=0pt,minimum size=.175cm] (0) at (0,0) {{}};
\node[circle,draw=black!100,fill=red!50, label=left:,inner sep=0pt,minimum size=.175cm] (a) at (-.75,.75) {{}};
\node[circle,draw=black!100,fill=orange!50, label=right:,inner sep=0pt,minimum size=.175cm] (b) at (.75,.75) {{}};
\node[circle,draw=black!100,fill=darkgreen!50, label=left:,inner sep=0pt,minimum size=.175cm] (1l) at (-.75,1.5) {{}};
\node[circle,draw=black!100,fill=blue!50, label=right:,inner sep=0pt,minimum size=.175cm] (1r) at (.75,1.5) {{}};
\node[circle,draw=black!100, label=above:,inner sep=0pt,minimum size=.175cm] (new1) at (0,2.25) {{}};

\path (new1) edge[-] node {{}} (1l);
\path (new1) edge[-] node {{}} (1r);
\path (1l) edge[-] node {{}} (a);
\path (1r) edge[-] node {{}} (b);
\path (a) edge[-] node {{}} (0);
\path (b) edge[-] node {{}} (0);

\path (1l) edge[<->,dashed,gray] node {{}} (b);
\path (1r) edge[<->,dashed,gray] node {{}} (a);

\end{tikzpicture}
\end{center}
\end{minipage}

\end{center}

\caption{Compatibility frames realizing ortholattices.}\label{OrthoFig}
\end{figure}

\subsection{Frames for Heyting and Boolean algebras}\label{BooleHeyt}

We conclude this section with sufficient conditions on $\comp$ for realizing Heyting and Boolean algebras, weaker than those in Proposition \ref{PWS} (where $\comp$ was a preorder or the identity relation). First, we define some auxiliary notions.

\begin{definition} Given a compatibility frame $(X,\comp)$ and $x,y\in X$:
\begin{enumerate}
\item $x$ \textit{pre-refines $y$}, written $x\sqsubseteq_{pr} y$, if for all $z\in X$, $z\comp x$ implies $z\comp y$;
\item $x$ \textit{post-refines $y$}, written $x\sqsubseteq_{po} y$, if for all $z\in X$, $x\comp z$ implies $y\comp z$;
\item $x$ \textit{refines} $y$, written $x\sqsubseteq y$, if $x$ pre-refines and post-refines $y$;
\item $x$ is \textit{compossible with $y$} if there is a $w\in X$ that refines $x$ and pre-refines $y$.
\end{enumerate} 
\end{definition}

\noindent Note that if $\comp$ is symmetric, then pre-refinement and post-refinement are equivalent, and $x$ is compossible with $y$ just in case they have a common refinement.

\begin{lemma}\label{DownsetLem} For any compatibility frame $(X,\comp)$, $\sqsubseteq_{pr}$ and $\sqsubseteq_{po}$ are preorders on $X$. Moreover, each $c_\comp$-fixpoint is a $\sqsubseteq_{pr}$-downset.
\end{lemma}
\begin{proof} The preorder part is obvious. Next suppose $A$ is a $c_\comp$-fixpoint, $x\in A$, and $y\sqsubseteq_{pr} x$. Toward showing that $y\in A$, consider any $y'\comp y$. Then since $y\sqsubseteq_{pr} x$, we have $y'\comp x$, so taking $y''=x$, we have shown that for every $y'\comp y$ there is a $y''\compflip y'$ with $y''\in A$. Since $A$ is a $c_\comp$-fixpoint, it follows that $y\in A$.\end{proof}

\begin{definition} A \textit{compossible compatibility frame} is a compatibility frame $(X,\comp)$ in which for any $x,y\in X$, if $x\comp y$, then $x$ is compossible with $y$.
\end{definition}

\begin{theorem}\label{HeytBool}$\,$ 
\begin{enumerate}
\item\label{HeytBool2} If $(X,\comp)$ is a compossible compatibility frame, then $\lat(X,\comp)$ is a Heyting algebra with $\to$ defined by $A\to B=\{x\in X\mid \forall y\sqsubseteq_{pr} x \, (y\in A\Rightarrow y\in B)\}$, and $\neg_\comp$ is the Heyting pseudocomplementation.
\item\label{HeytBool3} If $(X,\comp)$ is a compossible symmetric compatibility frame, then $\lat(X,\comp)$ is the Boolean algebra $\mathcal{RO}(X,\sqsubseteq)$, and $\neg_\comp$ is the Boolean complementation.
\end{enumerate}
\end{theorem}

\begin{proof} For part (\ref{HeytBool2}), recall that a \textit{nucleus} on a Heyting algebra $H$ is a closure operator that is also multiplicative, i.e., for all $x,y\in H$, we have $c(x)\wedge c(y)\leq c(x\wedge y)$. This follows from $c$ being a closure operator such that for all $a,b\in H$, $a\wedge c(b) \leq c(a\wedge b)$. For then setting $a=c(x)$ and $b=y$, $c(x)\wedge c(y)\leq c(c(x)\wedge y)=c(y\wedge c(x))$. Then setting $a=y$ and $b=x$, we have $y\wedge c(x)\leq c(y\wedge x)$, in which case monotonicity and idempotence yield $c(y\wedge c(x))\leq c(c(y\wedge x))= c(y\wedge x)=c(x\wedge y)$.
Thus, combining the two long strings of equations, $c(x)\wedge c(y)\leq c(x\wedge y)$.

Now it is well known that the fixpoints of a nucleus $j$ on a (complete) Heyting algebra $H$ form a (complete) Heyting algebra $H_j$ under the restricted lattice order (see, e.g., \cite[p.~71]{Dragalin1988}), where for $a,b\in H_j$, $a\to_{H_j} b = a\to_H b$.

Let $H$ be the Heyting algebra of all $\sqsubseteq_{pr}$-downsets. By the results above, to prove the first part of (\ref{HeytBool2}), it suffices to show that $c_\comp$ restricted to $H$ is a nucleus, for which it suffices to show that for all  $A,B\in H$, $A\cap c_\comp (B)\subseteq c_\comp(A\cap B)$. Suppose $x\in A\cap c_\comp (B)$ but $x\not\in c_\comp(A\cap B)$. Since $x\not\in c_\comp(A\cap B)$, there is a $y\comp x$ such that $(\star)$ for all $y'\compflip y$, we have $y'\not\in A\cap B$. Since $y\comp x$, by compossibility there is a $z$ that refines $y$ and pre-refines $x$. Since $z$ pre-refines $x$,  $z\comp x$. Then since $x\in c_\comp (B)$, there is a $z'\compflip  z$ with $z'\in B$.  Since $z\comp z'$, there is a $w$ that refines $z$ and pre-refines $z'$. Since $w$ pre-refines $z'$, from $z'\in B$ we have $w\in B$ by Lemma \ref{DownsetLem}. Since $w$ pre-refines $z$ and $z$ pre-refines $x$,  $w$ pre-refines $x$, so $x\in A$ implies $w\in A$. Thus, $w\in A\cap B$. Moreover, since $w$ post-refines $z$ and $z$ post-refines $y$, $w$ post-refines $y$. Hence $y\comp w$. But then by $(\star)$, $w\not\in A\cap B$. This is a contradiction. Finally, for the claim about $\neg_\comp$, observe that for any $\sqsubseteq_{pr}$-downset $A$, we have $A\to \varnothing=\{x\in X\mid \forall y \sqsubseteq_{pr}x\;\, y\not\in A\}=\neg_\comp A$; for the second equality, the right-to-left inclusion uses that $y \sqsubseteq_{pr}x$ implies $y\comp x$, while the left-to-right inclusion uses the assumption of compossibility.
 
 For part (\ref{HeytBool3}), to see that $\lat(X,\comp)$ is isomorphic to $\mathcal{RO}(X,\sqsubseteq)$, by Proposition \ref{RegOpen} it suffices to show that for all $x,y\in X$, we have $x\comp y$ iff there is a $z\in X$ such that $z\sqsubseteq x$ and $z\sqsubseteq y$. From left to right, if $x\comp y$, then since $(X,\comp)$ is a compossible compatibility frame, there is a $z$ that refines $x$ and pre-refines $y$, which implies that $z$ refines $y$ by the symmetry of $\comp$. From right to left, $z\sqsubseteq x$ implies $x\comp z$, which with $z\sqsubseteq y$ implies $x\comp y$. Finally, for the claim about $\neg_\comp$, recall that $\neg A$ in $\mathcal{RO}(X,\sqsubseteq)$  is $\{x\in X\mid \forall y\sqsubseteq x\;\, y\not\in A\}$, which is equal to $\neg_\comp A$ by reasoning analogous to that at the end of the previous paragraph.\end{proof}

\section{From lattices to compatibility frames}\label{LatticesToFrames}

\subsection{Representation of special complete lattices via join-dense sets}\label{SpecialRepSec}

In this section, we give an economical representation of certain complete lattices $L$ using compatibility frames  based on a join-dense set $\mathrm{V}$ of nonzero elements of $L$, so the frame representing $L$ is smaller than $L$. In all our figures, the frame can be seen as obtained from the lattice via this representation. 

Let $CJI(L)$ be the set of completely join-irreducible elements of $L$. Recall that a complete Heyting algebra is \textit{perfect} if $CJI(L)$ is join-dense in $L$. The standard representations of complete perfect Heyting algebras and  complete ortholattices drop out of the representation in this section as special cases. 

\subsubsection{Suitable compatibility relations}

Before giving the crucial definition of the compatibility relation used in our representation (Definition \ref{KeyDef}(\ref{KeyDef1}) below), it is instructive to distill conditions on a compatibility relation sufficient for a successful representation.

\begin{definition} Let $L$ be a lattice and $\mathrm{V}$ a set of elements of $L$. A binary relation $\comp$ on $\mathrm{V}$ is \textit{suitable for $L$} if $\comp$ is reflexive and:
\begin{enumerate}
\item for $a\in \mathrm{V}$ and $b\in L$, if $a\not\leq b$, then $\exists a'\comp a$ $\forall a''\compflip a'$  $a''\not\leq b$; 
\item if $a\in \mathrm{V}$, $B$ is a $c_\comp$-fixpoint of $(\mathrm{V},\comp)$ with a join $b$ in $L$, and $a\leq b$, then~$a\in B$.
\end{enumerate}
\end{definition}
\noindent We call (i) and (ii) the first and second \textit{suitability conditions}, respectively.

\begin{proposition}\label{SpecialRep} Let $L$ be a lattice, $\mathrm{V}$ a join-dense set of elements of $L$, and $\comp$ a reflexive relation on $\mathrm{V}$. For $b\in L$, define $\varphi(b)=\{x\in \mathrm{V}\mid x\leq b\}$.
\begin{enumerate}
\item\label{SpecialRep1} If $\comp$ satisfies the first suitability condition, then $\varphi$ embeds $L$ into $\lat(\mathrm{V},\comp)$.
\item\label{SpecialRep2} If $L$ is complete and $\comp$ suitable, $\varphi$ is an isomorphism from $L$ to $\lat(\mathrm{V},\comp)$.
\end{enumerate}
\end{proposition}

\begin{proof} For part (\ref{SpecialRep1}), clearly $\varphi$ is order-preserving: if $a\leq b$, then $\varphi(a)\subseteq \varphi(b)$. Moreover, $\varphi$ is order-reflecting:  since $\mathrm{V}$ is join-dense in $L$, we have $a=\bigvee A$ for some $A\subseteq \mathrm{V}$, so $a\not\leq b$ implies that for some $a_0\in A$, we have $a_0\not\leq b$, so $a_0\in \varphi(a)$ but $a_0\not\in\varphi(b)$, and hence $\varphi(a)\not\subseteq\varphi(b)$.

Next we claim $\varphi(b)$ is a $c_\comp$-fixpoint. Suppose for $a\in \mathrm{V}$ that $a\not\in \varphi(b)$, so $a\not\leq b$. Then by the first suitability condition, there is an $a'\comp a$ such that for all $a''\compflip a'$, $a''\not\leq b$ and hence $a''\not\in\varphi(b)$. Thus, $\varphi(b)$ is a $c_\comp$-fixpoint.

For (\ref{SpecialRep2}),  $\varphi$ is surjective. Let $B$ be a $c_\comp$-fixpoint and $b=\bigvee B$. We claim  $B=\varphi(b)$. For $B\subseteq\varphi(b)$, if $b_0\in B$, then $b_0\leq b$, so $b_0\in \varphi(b)$. For $B\supseteq\varphi(b)$, suppose $a\in \varphi(B)$, so $a\leq b$. Then by the second suitability condition, $a\in B$.\end{proof}

Our strategy for defining suitable compatibility relations on a join-dense set $\mathrm{V}$ of elements of $L$ will be to assume that $L$ comes equipped with an anti-inflationary operation~$\neg$, e.g., as in the case of an ortholattice with orthcomplementation $\neg$ or a p-algebra or Heyting algebra with pseudocomplementation $\neg$ or a lattice to which we have added an anti-inflationary operation $\neg$ as in Fig.~\ref{FirstFig}. We will then use $\neg$ to define a compatibility relation $\comp^\neg_\mathrm{V}$ on $\mathrm{V}$. 

\subsubsection{The first suitability condition}

In the case of an ortholattice, our defined compatibility relation will be equivalent to $x\comp y$ if $y\not\leq\neg x$ (Lemma \ref{KeyDefApplies}(\ref{KeyDefApplies1})). To see why this compatibility relation satisfies the first suitability condition, the following concept is useful.

\begin{definition}\label{EscapeDef}  Let $L$ be a lattice equipped with a unary operation $\neg$ and $\mathrm{V}$ a set of elements of $L$. Given $a,b\in L$, we say that \textit{$a$ escapes $b$ in $\mathrm{V}$ with $\neg$} if there is some $c\in \mathrm{V}$ such that $a\not\leq \neg c$ but $b\leq\neg c$. \end{definition}

\begin{lemma}\label{OrthoEscape} Let $L$ be an ortholattice, $\mathrm{V}$ a join-dense set of elements of $L$, and $a,b\in L$. If $a\not\leq b$, then $a$ escapes $b$ in $\mathrm{V}$ with the orthocomplementation $\neg$.
\end{lemma}

\begin{proof} Suppose $a\not\leq b$, so $\neg b\not\leq\neg a$. Since $\mathrm{V}$ is join-dense, we have $\neg b=\bigvee C$ for some $C\subseteq \mathrm{V}$. Then $\neg b\not\leq \neg a$ implies $c\not\leq\neg a$ for some $c\in C$. From $c\not\leq\neg a$ we have $a\not\leq\neg c$, and from $c\in C$ we have $c\leq\neg b$ and hence $b\leq\neg c$.
\end{proof}
\noindent Using Lemma \ref{OrthoEscape}, it is easy to see that the relation $\comp$ on ortholattices defined by $x\comp y$ if $y\not\leq \neg x$ satisfies the first suitability condition.

\begin{figure}[h]
\begin{center}
\begin{minipage}{1.5in}
\begin{center}
\begin{tikzpicture}[->,>=stealth',shorten >=1pt,shorten <=1pt, auto,node
distance=2cm,thick,every loop/.style={<-,shorten <=1pt}]
\tikzstyle{every state}=[fill=gray!20,draw=none,text=black]

\node[circle,draw=black!100, label=right:$$,inner sep=0pt,minimum size=.175cm] (n1) at (0,0) {{}};
\node[circle,fill=darkgreen!50,draw=black!100, label=right:$a$,inner sep=0pt,minimum size=.175cm] (nx) at (1,-.5) {{}};
\node[circle,fill=blue!50,draw=black!100, label=right:$b$,inner sep=0pt,minimum size=.175cm] (ny) at (1,-1.5) {{}};
\node[circle,fill=red!50,draw=black!100, label=right:$$,inner sep=0pt,minimum size=.175cm] (nz) at (-1,-1) {{}};
\node[circle,draw=black!100, label=right:$$,inner sep=0pt,minimum size=.175cm] (n0) at (0,-2) {{}};
\path (nx) edge[-] node {{}} (n1);
\path (nx) edge[-] node {{}} (ny);
\path (n1) edge[-] node {{}} (nz);
\path (ny) edge[-] node {{}} (n0);
\path (nz) edge[-] node {{}} (n0);

\path (ny) edge[->,dashed,gray] node {{}} (nz);

\path (nx) edge[->,dashed,gray] node {{}} (n0);

\path (nz) edge[->,dashed,gray] node {{}} (nx);

\end{tikzpicture}
\end{center}
\end{minipage}\begin{minipage}{1.5in}
\begin{center}
\begin{tikzpicture}[->,>=stealth',shorten >=1pt,shorten <=1pt, auto,node
distance=2cm,thick,every loop/.style={<-,shorten <=1pt}]
\tikzstyle{every state}=[fill=gray!20,draw=none,text=black]

\node[circle,draw=black!100, label=right:$$,inner sep=0pt,minimum size=.175cm] (n1) at (0,0) {{}};
\node[circle,fill=darkgreen!50,draw=black!100, label=left:$a$,inner sep=0pt,minimum size=.175cm] (nx) at (0,-.75) {{}};
\node[circle,fill=blue!50,draw=black!100, label=left:$b$,inner sep=0pt,minimum size=.175cm] (ny) at (0,-1.5) {{}};
\node[circle,draw=black!100, label=right:$$,inner sep=0pt,minimum size=.175cm] (n0) at (0,-2.25) {{}};
\path (nx) edge[-] node {{}} (n1);
\path (nx) edge[-] node {{}} (ny);
\path (ny) edge[-] node {{}} (n0);

\path (nx) edge[->,bend left=25,dashed,gray] node {{}} (n0);
\path (ny) edge[->,bend right=25,dashed,gray] node {{}} (n0);

\end{tikzpicture}
\end{center}
\end{minipage}
\end{center}
\caption{Lattices with $\neg$ (dashed arrows) in which $a\not\leq b$ but $a$ cannot escape $b$.}\label{NoEscape}
\end{figure}
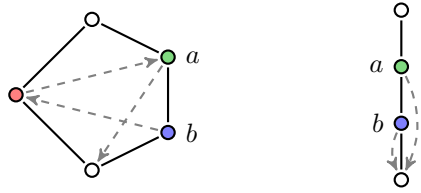

However, in lattices with $\neg$ that are not ortholattices, we can have $a\not\leq b$ while $a$ cannot escape $b$, as shown with two examples in Fig.~\ref{NoEscape}.

In order to deal with lattices in which $a\not\leq b$ does not imply that $a$ can escape $b$, we introduce the key definition of this section.

\begin{definition}\label{KeyDef} Let $L$ be a lattice, $\neg$ an anti-inflationary operation on $L$, and $\mathrm{V}$ a set of nonzero elements of $L$. 
\begin{enumerate}
\item\label{KeyDef1} Define $\comp^\neg_\mathrm{V}$ on $\mathrm{V}$ by: $x\comp^\neg_\mathrm{V} y$ if both $y\not\leq\neg x$ and for all $z\in L$, if $y\leq z$ but $x\not\leq z$, then $x$ escapes $z$ in $\mathrm{V}$ with $\neg$.
\item Given $a\in\mathrm{V}$ and $b\in L$, we say that \textit{$a$ compatibly escapes $b$ in $\mathrm{V}$ with $\neg$} if there is some $c\in \mathrm{V}$ such that $c\comp^\neg_\mathrm{V} a$ and $b\leq\neg c$.
\item $L$ \textit{has compatible escape with $\neg$ in $\mathrm{V}$} if  for all $a\in\mathrm{V}$ and $b\in L$, if $a$ escapes $b$ in $\mathrm{V}$ with $\neg$, then $a$ compatibly escapes $b$ in $\mathrm{V}$ with $\neg$.
\end{enumerate}
\end{definition}
The following is easy to check. 
\begin{lemma}\label{3CompLem} Under the assumptions of Definition \ref{KeyDef}, \textnormal{(i)} $\comp^\neg_\mathrm{V}$ is reflexive, \textnormal{(ii)}~$x\leq y$ implies $x\comp^\neg_\mathrm{V}y$, and \textnormal{(iii)} $x\comp^\neg_\mathrm{V}y\leq y'$ implies $x\comp^\neg_\mathrm{V}y'$.
\end{lemma}

It will turn out (Proposition \ref{CompEscapeWorks}) that $L$ having compatible escape ensures that $\comp^\neg_\mathrm{V}$ satisfies the first suitability condition. First, we show that ortholattices and Heyting algebras are alike in having compatible escape, and  $\comp^\neg_\mathrm{V}$ reduces to familiar relations in ortholattices and complete perfect Heyting algebras.

\begin{proposition}\label{KeyDefApplies} Let $L$ be a lattice and $\mathrm{V}$ a join-dense set of nonzero elements.
\begin{enumerate}
\item\label{KeyDefApplies1} If $L$ is an ortholattice, then $x\comp^\neg_\mathrm{V}y$ iff $y\not\leq\neg x$ for $x,y\in \mathrm{V}$; hence $L$ has compatible escape in $\mathrm{V}$ with $\neg$.
\item\label{KeyDefApplies2} If $L$ is a p-algebra, then $L$ has compatible escape in $\mathrm{V}$ with $\neg$.
\item\label{KeyDefApplies3} If $L$ is a complete Heyting algebra and $\mathrm{V}=CJI(L)$, then $x\comp^\neg_\mathrm{V}y$ iff $x\leq y$ for $x,y\in\mathrm{V}$.
\end{enumerate}

\end{proposition}
\begin{proof} For part (\ref{KeyDefApplies1}), by definition, $x\comp^\neg_\mathrm{V} y$ implies  $y\not\leq\neg x$.  Conversely, suppose $y\not\leq\neg x$. To show  $x\comp^\neg_\mathrm{V} y$, suppose $y\leq z$ but $x\not\leq z$.  From $x\not\leq z$, it follows by Lemma \ref{OrthoEscape} that $x$ escapes $z$ in $\mathrm{V}$ with $\neg$. This shows  $x\comp^\neg_\mathrm{V} y$.

For (\ref{KeyDefApplies2}), suppose $a$ escapes $b$ in $\mathrm{V}$, so for some $c\in \mathrm{V}$, $a\not\leq\neg c$ but $b\leq\neg c$. Since $a\not\leq \neg c$ and $\neg$ is pseudocomplementation, we have $a\wedge c\neq 0$. From $b\leq\neg c$, we have $b\leq\neg (a\wedge c)$.  Since $\mathrm{V}$ is join-dense in $L$, there is a $d\in \mathrm{V}$ such that  $d\leq a\wedge c$, so $\neg (a\wedge c)\leq\neg d$. Then since $b\leq\neg (a\wedge c)$, we have $b\leq\neg d$. Moreover, by Lemma \ref{3CompLem}(ii), $d\leq a$ implies $d\comp^\neg_\mathrm{V} a$. Hence we have found a $d\comp^\neg_\mathrm{V} a$ such that $b\leq\neg d$. Thus, $a$ compatibly escapes $b$ in $\mathrm{V}$ with $\neg$.

For (\ref{KeyDefApplies3}), given Lemma \ref{3CompLem}(ii), we need only show that $x\comp^\neg_\mathrm{V}y$ implies $x\leq y$. Suppose $x\comp^\neg_\mathrm{V}y$ but $x\not\leq y$. Let $w=\bigvee\{z\in L\mid y\leq z,x\not\leq z\}$. Since $x\in CJI(L)$ and $L$ is a complete Heyting algebra, $x$ is completely join-prime, which implies $x\not\leq w$. Then since $x\comp^\neg_\mathrm{V}y$ and $y\leq w$,  $x$ escapes $w$, so for some $v\in\mathrm{V}$, $x\not\leq\neg v$ but $w\leq\neg v$ and thus $y\leq\neg v$. Hence $x\not\leq v$, for otherwise $\neg v\leq \neg x$ and so $y\leq\neg x$, contradicting $x\comp^\neg_\mathrm{V}y$. From $x\not\leq w$ and $x\not\leq v$, we have $x\not\leq w\vee v$, while  $y\leq w\vee v$. Since $w\leq\neg v$, we also have $w<w\vee v$. But together $y\leq w\vee v$, $x\not\leq w\vee v$,  and $w<w\vee v$ contradict the definition of $w$.\end{proof}

Crucially, there are other lattices with compatible escape besides ortholattices and p-algebras, as in $\mathbf{N}_5$ equipped with the $\neg$ operation in Fig.~\ref{NoEscape} (left).

\begin{proposition}\label{CompEscapeWorks} Let $L$ be a lattice, $\neg$  an anti-inflationary operation on $L$, and $\mathrm{V}$ a set of nonzero elements of $L$. If $L$ has compatible escape in $\mathrm{V}$ with $\neg$, then $\comp^\neg_\mathrm{V}$ satisfies the first suitability condition. 
\end{proposition}

\begin{proof} For $a\in \mathrm{V}$ and $b\in L$, suppose $a\not\leq b$.

Case 1: $a$ escapes $b$. Then since $L$ has compatible escape in $\mathrm{V}$ with $\neg$, $a$ compatibly escapes $b$ using  some $c\in\mathrm{V}$, so $c\comp^\neg_\mathrm{V} a$ and $b\leq \neg c$. Let $a'=c$, so $a'\comp^\neg_\mathrm{V} a$. Suppose $a''\compflip^\neg_\mathrm{V} a'$, so $a''\not\leq\neg c$. Then since $b\leq\neg c$, we have $a''\not\leq b$. 

Case 2: $a$ does not escape $b$. Then let $a'=a$, so $a'\comp^\neg_\mathrm{V} a$. Suppose $a''\compflip^\neg_\mathrm{V} a'$. If $a''\leq b$, then since $a\not\leq b$, from $a\comp^\neg_\mathrm{V} a''$ it follows that $a$ escapes $b$, contradicting the assumption of the case. Hence $a''\not\leq b$.\end{proof}

\begin{corollary} Let $L$ be a lattice and $\mathrm{V}$ a join-dense set of nonzero elements.
\begin{enumerate}
\item If $L$ is an ortholattice, then $\comp^\neg_\mathrm{V}$ satisfies the first suitability condition.
\item If $L$ is a p-algebra, then $\comp^\neg_\mathrm{V}$ satisfies the first suitability condition.
\end{enumerate}
\end{corollary}

\subsubsection{The second suitability condition}

We can treat the second suitability condition more quickly. 

\begin{proposition}\label{SecondSuit} 
\begin{enumerate}
\item If $L$ is a complete ortholattice and $\mathrm{V}$ a join-dense set of nonzero elements of $L$, then $\comp^\neg_\mathrm{V}$ satisfies the second suitability condition.
\item If $L$ is a complete Heyting algebra and $\mathrm{V}=CJI(L)$ is join-dense in $L$, then $\comp^\neg_\mathrm{V}$ satisfies the second suitability condition.
\end{enumerate}
\end{proposition}

\begin{proof} Suppose $L$ is an ortholattice, $a\in \mathrm{V}$, $B$ is a $c_{\comp^\neg_\mathrm{V}}$-fixpoint, and $a\leq b=\bigvee B$. Suppose $a'\comp^\neg_\mathrm{V} a$, so $a\not\leq\neg a'$ and hence $b\not\leq\neg a'$. Then for some $b_0\in B$, we have $b_0\not\leq \neg a'$. Now let $a''=b_0$, so $a''\in B$. Since $a''\not\leq\neg a'$, by Proposition \ref{KeyDefApplies}(\ref{KeyDefApplies1}) we have $a'\comp ^\neg_\mathrm{V} a''$. Given that $B$ is a $c_{\comp^\neg_\mathrm{V}}$-fixpoint, this shows  $a\in B$.

Suppose $L$ is a complete Heyting algebra, $a\in \mathrm{V}$, $B$ is a $c_{\comp^\neg_\mathrm{V}}$-fixpoint, and $a\leq \bigvee B$. Since $L$ is a complete Heyting algebra and $\mathrm{V}=CJI(L)$, $a\in \mathrm{V}$ implies that $a$ is completely join-prime. Hence $a\leq b_0$ for some $b_0\in B$.  Now consider any $a'\comp^\neg_\mathrm{V} a$. Let $a''=b_0$, so $a''\in B$. Since $a'\comp^\neg_\mathrm{V}  a\leq a''$, by Lemma \ref{3CompLem}(iii) we have $a'\comp^\neg_\mathrm{V}a''$. Given that $B$ is a $c_{\comp^\neg_\mathrm{V}}$-fixpoint, this shows $a\in B$.\end{proof}

\subsubsection{Representation theorem}\label{FirstRepThmSec}

Combining Propositions \ref{SpecialRep}-\ref{SecondSuit}, we have the following representation theorem.

\begin{theorem}\label{RepThm1} Let $L$ be a complete lattice and $\mathrm{V}$ a join-dense set of nonzero elements of $L$.
\begin{enumerate}
\item If $L$ is an ortholattice, then $L$ is isomorphic to $\lat(\mathrm{V},\comp^\neg_\mathrm{V})$.\footnote{The orthocomplementation of $L$ is represented as $\neg_{\comp^\neg_\mathrm{V}}$ by Propositions \ref{EconRepNeg} and \ref{EconRepNegOrtho}.}
\item If $L$ is a  Heyting algebra and $\mathrm{V}=CJI(L)$, $L$ is isomorphic to $\lat(\mathrm{V},\comp^\neg_\mathrm{V})$.
\item\label{RepThm1-3} If $\neg$ is an anti-inflationary operation on $L$ such that $L$ has compatible escape in $\mathrm{V}$ with $\neg$, and $\comp^\neg_\mathrm{V}$ satisfies the second suitability condition, then $L$ is isomorphic to $\lat(\mathrm{V},\comp^\neg_\mathrm{V})$.  
\end{enumerate}
\end{theorem}

Part (\ref{RepThm1-3}) applies to many non-ortholattice and non-Heyting examples, as in Fig.~\ref{FirstFig}, but its precise scope is an open question.

\begin{question}\label{CharQues} For which complete lattices $L$ is there a set $\mathrm{V}$ of nonzero elements and an anti-inflationary $\neg $ on $L$ such that $L$ is isomorphic to $\lat(\mathrm{V},\comp^\neg_\mathrm{V})$?
\end{question} 
\noindent The Jupyter notebook cited in \S~\ref{Intro} verifies that every lattice $L$  up to size 8 is such an $L$. It also verifies the following for every lattice up to size 16 (cf.~Footnote~\ref{Bipartite}).

\begin{conjecture}\label{Conj} For any nondegenerate finite lattice $L$, there is a compatibility frame $(X,\comp)$ with $|X|< |L|$ such that $L$ is isomorphic to $\lat(X,\comp)$.
\end{conjecture}

Finally, Appendix \ref{NegAppendix} modifies (\ref{RepThm1-3}) to represent not only $L$ but also~$(L,\neg)$.

\subsection{Representation of arbitrary complete lattices}\label{ArbCompLat}

In this section, we turn to the representation of arbitrary complete lattices. Instead of representing a lattice $L$ using special elements of $L$ as in \S~\ref{SpecialRepSec}, here we use a potentially less economical representation in terms of pairs of elements, as in the birelational representations of complete lattices in \cite{Allwein2001,BH2019,Holliday2021,Massas2020}.

\begin{definition}\label{Good} Let $L$ be a lattice and $P$ a set of pairs of elements of $L$. Define a binary relation $\comp$ on $P$ by 
$(a,b)\comp (c,d)$ if $c\not\leq b$. Then we say $P$ is \textit{separating} if for all $a,b\in L$:
\begin{enumerate}
\item\label{Good2} if $a\not\leq b$, then there is a $(c,d)\in P$ with $c\leq a$ and $c\not\leq b$;
\item\label{Good3} for all $(c,d)\in P$, if $c\not\leq b$, then there is a $(c',d')\comp (c,d)$ such that for all $(c'',d'')\compflip (c',d')$, we have $c''\not\leq b$.
\end{enumerate}
\end{definition}

\begin{proposition}\label{CompRep} Let $L$ be a lattice and $P$ a separating set of pairs of elements of $L$. For $a\in L$, define $\varphi(a)= \{(x,y) \in P\mid x\leq a\}$. Then:
\begin{enumerate}
\item\label{CompRep1} $\varphi$ is a lattice embedding of $L$ into $\lat(P,\comp)$;
\item\label{CompRep2} if $L$ is complete, then $\varphi$ is an isomorphism from $L$ to $\lat(P,\comp)$.
\end{enumerate}
\end{proposition}

\begin{proof} For part (\ref{CompRep1}), condition (\ref{Good3}) of Definition \ref{Good} implies that $\varphi(b)$ is a $c_\comp$-fixpoint for each $b\in L$. Clearly $\varphi$ preserves meet.  For join, if $(x,y)\in \varphi(a\vee b)$ and $(x',y')\comp (x,y)$, then $a\not\leq y'$ or $b\not\leq y'$, so (\ref{Good2}) of Definition \ref{Good} yields an $(x'',y'')\in \varphi(a)\cup\varphi(b)$ with $(x',y')\comp (x'',y'')$.  It also ensures that $\varphi$ is injective.

For part (\ref{CompRep2}), we claim $\varphi$ is surjective. Given a $c_{\comp}$-fixpoint $A$, define $a=\bigvee\{a_i\mid \exists b_i:(a_i,b_i)\in A\}$. We claim $A=\varphi(a)$. For $A\subseteq\varphi(a)$, suppose $(a_i,b_i)\in A$. Then by definition of $a$,  $a_i\leq a$, so $(a_i,b_i)\in  \varphi(a)$. For $A\supseteq\varphi(a)$, suppose $(c,d)\in \varphi(a)$, so $c\leq a$. Since $A$ is a $c_{\comp}$-fixpoint, to show $(c,d)\in A$, it suffices to show that for every $(c,'d')\comp (c,d)$ there is a $(c'',d'')\compflip (c',d')$ with $(c'',d'')\in A$. Suppose $(c,'d')\comp (c,d)$, so $c\not\leq d'$, which with $c\leq a$ implies $a\not\leq d'$. Then for some $(a_i,b_i)\in A$, we have $a_i\not\leq d'$.  Setting $(c'',d'')=(a_i,b_i)$, from $a_i\not\leq d'$ we have $(c',d')\comp (c'',d'')$, and $(c'',d'')\in A$, so we are done.
\end{proof}

Every lattice has a separating set of pairs, and for ortholattices and Heyting algebras we can cut down the sets of pairs, as in Proposition \ref{GoodEx}. Compare part (\ref{GoodEx1})  to the representation of a complete lattice by join-dense and meet-dense sets in \cite{Markowsky1973,Markowsky1975} or in the fundamental theorem of concept lattices \cite[Thm.~3.9]{Davey2002}.

\begin{proposition}\label{GoodEx} Let $L$ be a lattice, $\mathrm{V}$ a join-dense set of elements of $L$, and $\Lambda$ a meet-dense set of elements of $L$. Then:
\begin{enumerate}
\item\label{GoodEx1} if $P_1$ is a subset of $P_0=\{(a,b) \mid a\in \mathrm{V},b\in \Lambda,a\not\leq b\}$ such that for each $a\in \mathrm{V}$ there is a $b\in\Lambda$ with $(a,b)\in P_1$, and for each $b\in\Lambda$ there is an $a\in \mathrm{V}$ with $(a,b)\in P_1$, then $P_1$ is separating;\footnote{\label{Bipartite}Here $P_1$ is an edge cover for the bipartite graph $(\mathrm{V},\Lambda, P_0)$. If for every finite lattice $L$, the smallest edge cover for the bipartite graph of its join- and meet-irreducibles has cardinality less than $|L|$ (as we verified for all $L$ with $|L|\leq 16$ using \cite{Gebhardt2020}), then Conjecture \ref{Conj} is true.}
\item\label{GoodEx2} if $L$ is an ortholattice, then $P_2=\{(a,\neg a)\mid a\in \mathrm{V}, a\neq 0\}$ is separating;
\item\label{GoodEx3} if $L$ is Heyting, then $P_3=\{(a,a\to b)\mid a\in \mathrm{V},b\in \Lambda ,a\not\leq b\}$ is separating.
\end{enumerate}
\end{proposition}

\begin{proof} In each case, that condition (\ref{Good2}) of Definition \ref{Good} is satisfied is obvious, so we focus on condition (\ref{Good3}). 

For $P_1$, suppose $(c,d)\in P_1$ and $c\not\leq b$. Then where $b=\bigwedge \{b_i\in\Lambda\mid i\in I\}$, there is some $b_i\in\Lambda$ such that $c\not\leq b_i$, but $b\leq b_i$, and some $a\in\mathrm{V}$ with $(a,b_i)\in P_1$. Let $(c',d')=(a,b_i)$, so $(c',d')\comp (c,d)$. Now consider any $(c'',d'')\in P_1$ with $(c',d')\comp (c'',d'')$. Then $c''\not\leq d'=b_i$, so $c''\not\leq b$. 

For the ortholattice case with $P_2$, suppose $(c,d)=(c,\neg c)\in P_2$ and $c\not\leq b$. Hence $\neg b\not\leq \neg c$. Then where $\neg b=\bigvee\{x_i\in\mathrm{V}\mid i\in I\}$, for some $i\in I$ we have $x_i\not\leq \neg c$, but $x_i\leq \neg b$, so  $c\not\not\leq \neg x_i$ and $b\leq\neg x_i$. Let $(c',d')=(x_i,\neg x_i)$. Since $c\not\leq \neg x_i$, we have $(c',d')\comp (c,d)$. Now consider any $(c'',d'')\in P_2$ with $(c',d')\comp (c'',d'')$. Then $c''\not\leq d'=\neg x_i$, which with $b\leq\neg x_i$ implies $c''\not\leq b$. 

For the Heyting case with $P_3$, suppose $(c,d)=(c,c\to e)\in P_3$ and $c\not\leq b$. Then where $b=\bigwedge \{b_i\in\Lambda\mid i\in I\}$, there is some $b_i\in\Lambda$ such that $c\not\leq b_i$, but $b\leq b_i$. From $c\not\leq b_i$, we have $c\not\leq c\to b_i$.  Then where $(c',d')=(c,c\to b_i)$, we have $(c',d')\in P_3$ and $(c',d')\comp (c,d)$. Now consider any $(c'',d'')\in P_3$ such that $(c',d')\comp (c'',d'')$, so $c''\not\leq d'=c\to b_i$. Then $c''\wedge c\not\leq b_i$, which with $b\leq b_i$ implies $c''\wedge c\not\leq b$ and hence $c''\not\leq b$.\end{proof}

By Propositions \ref{CompRep} and \ref{GoodEx}, $L$ embeds into $\lat(P_1,\comp)$.  As the image of the embedding is join-dense and meet-dense in $\lat(P_1,\comp)$, it follows that $\lat(P_1,\comp)$ is (up to isomorphism) the MacNeille completion of $L$ (see  \cite[Thm.~2.2]{Gehrke2005}), and similarly for $\lat(P_2,\comp)$ and $\lat(P_3,\comp)$ in the ortholattice\footnote{It is also easy to check that $\varphi$ respects the orthocomplementation: $\varphi(\neg a)=\neg_{\comp}\varphi(a)$.} and Heyting cases.

Recall from Propositions \ref{IsOrtho} and \ref{HeytBool} how symmetric, compossible, and symmetric compossible compatibility frames give rise to ortholattices, Heyting algebras, and Boolean algebras, respectively. We now prove a converse.

\begin{proposition}\label{GoodFr} Let $L$ be a lattice, $\mathrm{V}$ a join-dense set of elements of $L$, and $\Lambda$ a meet-dense set of elements of $L$. Let $P_1$, $P_2$, and $P_3$ be defined from $\mathrm{V}$ and $\Lambda$ as in Proposition \ref{GoodEx}. Then:
\begin{enumerate}
\item\label{GoodFr1} $(P_1,\comp)$ is a compatibility frame;
\item\label{GoodFr2} if $L$ is an ortholattice, then $(P_2,\comp)$ is a symmetric frame;
\item\label{GoodFr3} if $L$ is a Heyting algebra, then $(P_3,\comp)$ is a compossible frame;
\item\label{GoodFr4} if $L$ is a Boolean algebra, then $(P_2,\comp)$ is a symmetric compossible frame.
\end{enumerate}
\end{proposition}
\begin{proof}  For each part,  $(a,b)\in P_i$ implies $a\not\leq b$, so $(a,b)\comp (a,b)$. Hence $\comp$ is reflexive. For part (\ref{GoodFr2}), if $(a,\neg a)\comp (b,\neg b)$, so $b\not\leq \neg a$, then  $a\not\leq \neg b$, so $(b,\neg b)\comp (a,\neg a)$. Hence $\comp$ is symmetric.

For (\ref{GoodFr3}), toward showing that $(P_3,\comp)$ is a compossible compatibility frame, suppose $(a,a\to b)\comp (c,c\to d)$. We must show there is an $(x,x\to y)\in P_3$ that refines $(a,a\to b)$ and pre-refines $(c,c\to d)$. Since $(a,a\to b)\comp (c,c\to d)$, we have $c\not\leq a\to b$, so $c\wedge a\not\leq b$. Then where $c\wedge a=\bigvee \{x_i\in \mathrm{V}\mid i\in I\}$ and $b=\bigwedge\{b_k\in \Lambda\mid k\in K\}$, there are $i\in I$ and $k\in K$ such that $x_i\not\leq b_k$ and hence $x_i\not\leq x_i\to b_k$. We claim that $(x_i,x_i\to b_k)$ refines $(a,a\to b)$ and pre-refines $(c,c\to d)$. To see that $(x_i,x_i\to b_k)$ pre-refines $(a,a\to b)$ and $(c,c\to d)$, suppose $(w,w\to v)\comp (x_i,x_i\to b_k)$, so $x_i\not\leq w\to v$. Then since $x_i\leq a$, we have $a\not\leq w\to v$, so $(w,w\to v)\comp (a,a\to b)$; and since $x_i\leq c$, we have $c\not\leq w\to v$, so  $(w,w\to v)\comp (c,c\to d)$. To see that $(x_i,x_i\to b_k)$ post-refines $(a,a\to b)$, suppose $(x_i,x_i\to b_k)\comp (w,w\to v)$, so $w\not\leq x_i\to b_k$. It follows that $w\wedge x_i\not\leq b_k$ and hence $w\wedge x_i\not\leq b$, which with $x_i\leq a$ implies $w\wedge a\not\leq b$, so $w\not\leq a\to b$ and hence $(a,a\to b)\comp (w,w\to v)$.

For (\ref{GoodFr4}), symmetry follows from part (\ref{GoodFr2}). As for compossibility, if $(a,\neg a)\comp (b,\neg b)$, so $b\not\leq \neg a$, then $a\wedge b\neq 0$. Hence there is some nonzero $c\in\mathrm{V}$ with $c\leq a\wedge b$, which implies that $(c,\neg c)\in P_2$ refines both $(a,\neg a)$ and $(b,\neg b)$.\end{proof}

Combining Propositions \ref{CompRep}-\ref{GoodFr}, we have the following.

\begin{theorem}\label{CombinedThm} If $L$ is a lattice \textnormal{(}resp.~ortholattice, Heyting algebra, Boolean algebra\textnormal{)}, then $L$ embeds into the lattice of $c_\comp$-fixpoints of a compatibility frame \textnormal{(}resp.~symmetric, compossible, symmetric compossible compatibility frame\textnormal{)}, and if $L$ is complete, the embedding is an isomorphism.
\end{theorem}
\noindent Thus, compossible compatibility frames yield a semantics for intermediate logics as general as complete Heyting algebras in the ``semantic hierarchy'' of \cite{BH2019}.

Finally, we prove that we can also represent any complete lattice expanded with a protocomplementation using a compatibility frame.

\begin{theorem}\label{NegThm} For any bounded lattice $L$ equipped with a protocomplementation $\neg$, the expansion $(L,\neg)$ embeds into the lattice of $c_\comp$-fixpoints of a compatibility frame equipped with $\neg_{\comp}$, and if $L$ is complete, the embedding is an isomorphism.
\end{theorem}

\begin{proof} First, we claim $P=\{(a,b)\mid a,b\in L,a\not\leq b,\neg a\leq b\}$ is separating.  For part (\ref{Good2}) of Definition \ref{Good}, take $(c,d)=(a,\neg a)$. For (\ref{Good3}), suppose $(c,d)\in P$ and $c\not\leq b$. Let $(c',d')=(1,b)$. Since $b\neq 1$ and $\neg 1=0\leq b$, $(1,b)\in P$, and since $c\not\leq b$, $(c',d')\comp (c,d)$. Now consider any $(c'',d'')\in P$ with $(c',d')\comp (c'',d'')$. Then $c''\not\leq d'=b$, so (\ref{Good3}) holds. Thus, by Proposition \ref{CompRep}, $L$ embeds into $\mathfrak{L}(P,\comp)$ via the map $\varphi$, which is an isomorphism if $L$ is complete. Also observe that $\comp$ is reflexive on $P$. It only remains to show $\varphi(\neg a)=\neg_\comp\varphi(a)$. Suppose $(x,y)\in \varphi(\neg a)$, so $x\leq\neg a$, and $(x',y')\comp (x,y)$. If $x'\leq a$, then $\neg a\leq\neg x'$, which with $x\leq\neg a$ implies $x\leq \neg x'$, which with $\neg x'\leq y'$ implies $x\leq y'$, contradicting $(x',y')\comp (x,y)$. Thus, $x'\not\leq a$, so $(x',y')\not\in \varphi(a)$. Hence $(x,y)\in \neg_\comp \varphi(a)$. Conversely, let $(x,y)\in P\setminus\varphi(\neg a)$, so $x\not\leq \neg a$. Since $\neg 0=1$, it follows that $a\neq 0$, so $(a,\neg a)\in P$, and $(a,\neg a)\comp (x,y)$, so $(x,y)\not\in \neg_\comp\varphi(a)$.\end{proof}

Note that any bounded lattice can be equipped with the protocomplementation such that $\neg 0=1$ and $\neg a=0$ for $a\neq 0$ (in which case the set $P$ in the proof of Theorem \ref{NegThm} coincides with $P_0$ from Proposition \ref{GoodEx} where $\mathrm{V}$ and $\Lambda$ are the non-minimum and non-maximum elements of $L$, respectively), so Theorem \ref{NegThm} generalizes the part of Theorem \ref{CombinedThm} concerning bounded lattices.

\subsection{Representation of arbitrary lattices}\label{SubRepSec}

There is another way of representing any lattice $L$ as a sublattice of the lattice of $c_\comp$-fixpoints of a compatibility frame, which is now the canonical extension of $L$ (see \cite{Gehrke2001,Craig2014}) rather than its MacNeille completion.\footnote{We do not have space to discuss $(L,\neg)$ under this approach, so we save this for the future.} The sublattice can then be characterized in a simple way topologically.  This approach uses disjoint filter-ideal pairs and appears already in \cite{Craig2013}, building on \cite{Urquhart1978,Allwein1993}, though we use a different topology in order to generalize the choice-free Stone duality of \cite{BH2020}. Given a  lattice $L$, define $\mathsf{FI}(L)=(X,\comp)$ as follows: $X$ is the set of all pairs $(F,I)$ such that $F$ is a  filter in $L$, $I$ is an ideal in $L$, and $F\cap I=\varnothing$; and $(F,I)\comp (F',I')$ iff $I\cap F'=\varnothing$. Given $a\in L$, let $\widehat{a}=\{(F,I)\in X\mid a\in F\}$. Let $\mathsf{S}(L)$ be $\mathsf{FI}(L)$ endowed with the topology generated by $\{\widehat{a}\mid a\in L\}$.

In Appendix \ref{SubAppendix}, we prove the following (without choice). 

\begin{theorem}\label{EmbedThm} For any lattice \textnormal{(}resp.~bounded lattice\textnormal{)} $L$, the map $a\mapsto\widehat{a}$ is \textnormal{(i)} a lattice \textnormal{(}resp.~bounded lattice\textnormal{)} embedding  of $L$ into $\lat(\mathsf{FI}(L))$ and \textnormal{(ii)} an isomorphism from $L$ to the sublattice of $\lat(\mathsf{FI}(L))$ consisting of $c_\comp$-fixpoints that are compact open in the space $\mathsf{S}(L)$.
\end{theorem}

In Appendix \ref{SubAppendix} we also prove the following characterization of spaces equipped with a relation $\comp$ that are isomorphic to $\mathsf{S}(L)$ for some $L$ in a manner analogous to the characterization of UV-spaces in \cite{BH2020}. Let $X$ be a topological space and $\comp$ a reflexive relation on $X$. Let $\mathsf{COFix}(X,\comp)$ be the set of all compact open sets of $X$ that are also $c_\comp$-fixpoints. Given $U,V\in \mathsf{COFix}(X,\comp)$, $U\vee V =c_\comp(U\cup V)$. Given $x\in X$, let 
\begin{align*}
\mathsf{F}(x)&=\{U\in \mathsf{COFix}(X,\comp)\mid x\in U\} \\
\mathsf{I}(x)&=\{U\in \mathsf{COFix}(X,\comp)\mid \forall y\compflip x\;\, y\not\in U\}.\end{align*} 

\begin{proposition}\label{COFixProp} For any space $X$ and reflexive binary relation $\comp$ on $X$, there is a lattice $L$ such that $(X,\comp)$ and $\mathsf{S}(L)$ are homeomorphic as spaces and isomorphic as relational frames iff the following conditions hold for all $x,y\in X$: \textnormal{(i)} $x=y$ iff $(\mathsf{F}(x),\mathsf{I}(x))=(\mathsf{F}(y),\mathsf{I}(y))$\textnormal{;} \textnormal{(ii)} $\mathsf{COFix}(X,\comp)$ is closed under $\cap$ and $\vee$ and forms a basis for $X$\textnormal{;} \textnormal{(iii)} each disjoint filter-ideal pair from  $\mathsf{COFix}(X,\comp)$ is $(\mathsf{F}(x), \mathsf{I}(x))$ for some $x\in X$\textnormal{;} \textnormal{(iv)} $x\comp y$ iff $\mathsf{I}(x)\cap \mathsf{F}(y)= \varnothing$.
\end{proposition}

\section{Compatibility and accessibility frames}\label{ModalCase}

In this section, we extend the three representations from \S\S~\ref{SpecialRepSec}, \ref{ArbCompLat}, and \ref{SubRepSec} to lattices equipped with a modal operation $\Box$ (we defer other modalities not definable from $\Box$ to future work). First, we add an accessibility relation $R$ to compatibility frames and require that the standard modal operation $\Box_R$ sends $c_\comp$-fixpoints to $c_\comp$-fixpoints.  These frames are similar to the \textit{graph-based frames} of \cite[Definition 2]{Conradie2019}, which have been applied in \cite{Conradie2019b,Conradie2020,Conradie2021}. Conradie et al.~\cite[Theorem~1]{Conradie2019} use the filter-ideal frame $\mathsf{FI}(L)$ equipped with accessibility relations to prove completeness of the minimal non-distributive modal logic with respect to graph-based frames (compare our Theorem \ref{ModalEmbedding}); in addition, they treat Sahlqvist correspondence theory for graph-based frames. 

There are many related approaches to representing lattices with modalities in the literature (see, e.g., \cite{Orlowska2005,Gehrke2006,Conradie2016,Hartonas2018,Hartonas2019,Goldblatt2019,Dmitrieva2021} and references therein). The approach below using just two binary relations on a single set is of special interest to us as a non-classical generalization of classical ``possibility semantics''  for modal logic \cite{Humberstone1981,HollidayForthA,Holliday2021b,Yamamoto2017,Zhao2016}. Our motivation for such a non-classical generalization comes from a recent application of the approach below to modal ortholattices for natural language semantics \cite{HM2021}. 

All proofs in this section are deferred to Appendix \ref{NecAppendix}.

\begin{definition} A \textit{necessity lattice} is a pair $(L,\Box)$ where $L$ is a lattice and $\Box$ is a unary operation on $L$ that is multiplicative, i.e., $\Box (a\wedge b)=\Box a\wedge\Box b$ for all $a,b\in L$, and $\Box 1=1$ if $L$ contains a maximum element $1$.  We say $\Box$ is \textit{completely multiplicative} if for any $A\subseteq L$, if $\bigwedge A$ exists in $L$, then $\Box \bigwedge A= \bigwedge \{\Box a\mid a\in A\}$.
\end{definition}

\begin{definition}\label{CAframes} A \textit{compatibility and accessibility} (\textit{CA}) \textit{frame} is a triple ${(X,\comp , R)}$ such that $(X,\comp)$ is a compatibility frame  and $R$ is a binary relation on $X$ such that for any $A\subseteq X$, if $A$ is a $c_\comp$-fixpoint, then so is
\[\Box_R A=\{x\in X\mid R(x)\subseteq A\},\] where $R(x)=\{y\in X\mid xRy\}$.
\end{definition}

\noindent Stronger conditions on the interplay of $\comp$ and $R$ could be imposed (see \cite[Def.~2]{Conradie2019} and Proposition \ref{FO} below) but Definition \ref{CAframes} suffices for the following.

\begin{proposition} For any CA frame $(X,\comp, R)$, the pair $(\lat(X,\comp),\Box_R)$ is a complete necessity lattice with $\Box_R$ completely multiplicative.
\end{proposition}
\begin{proof} That $\lat(X,\comp)$ is a complete lattice is Corollary \ref{FrameToLate}. Recall that meet is intersection. Then the complete multiplicativity of $\Box_R$ is obvious.
\end{proof}

\begin{example} Fig.~\ref{CAFig} shows a CA frame (left) where a dotted line from $w$ to $v$ means $wRv$. Observe that $\Box_R(\{x\})=\{x\}$, $\Box_R(\{y,z\})=\{z\}$, and $\Box_R(\{z\})=\{z\}$. Thus, $\Box_R$ sends $c_\comp$-fixpoints to $c_\comp$-fixpoints. The $\Box_R$ operation on the lattice of $c_\comp$-fixpoints (right) is represented by the double-shafted arrows.
\end{example}

\begin{figure}[h]
\begin{center}
\begin{minipage}{1.75in}
\begin{center}
\begin{tikzpicture}[->,>=stealth,shorten >=1pt,shorten <=1pt, auto,node
distance=2cm,thick,every loop/.style={<-,shorten <=1pt}]
\tikzstyle{every state}=[fill=gray!20,draw=none,text=black]

\tikzset{every loop/.style={min distance=10mm,looseness=10}}

\node[label=center:$x$,inner sep=0pt,minimum size=.175cm] at (0,0) (D) {}; 
\node[label=center:$y$,inner sep=0pt,minimum size=.175cm] at (1,1) (F) {}; 
\node[label=center:$z$,inner sep=0pt,minimum size=.175cm] at (2,0) (H) {}; 

\path[-{Triangle[open]},draw,thick] (D) to node {{}}  (F);
\path[-{Triangle[open]},draw,thick] (F) to node {{}}  (H);

\path[loop left,->,dotted,draw,thick] (D) to node {{}}  (D);
\path[loop right,->,dotted,draw,thick] (H) to node {{}}  (H);
\path[bend left=40,->,dotted,draw,thick] (F) to node {{}}  (H);
\path[bend right=40,->,dotted,draw,thick] (F) to node {{}}  (D);
\path[loop above,->,dotted,draw,thick] (F) to node {{}}  (D);

\path[-, draw=red, opacity=0.5, thick, rounded corners]  (0, .4) -- (.4, .4) -- (.4, -.4) -- (-.4, -.4) -- (-.4, .4) -- (.4, .4) -- (0, .4); 

\path[-, draw=darkgreen, opacity=0.5, thick, rounded corners] (1.1, 1.5) -- (1.5, 1.5) -- (2.5, .5) -- (2.5, -.5) -- (1.5, -.5) -- (.5, .5)  -- (.5, 1.5) -- (1.5, 1.5) -- (1.1, 1.5); 

\path[-, draw=blue, opacity=0.5, thick, rounded corners] (2, .4) -- (2.4, .4) -- (2.4, -.4) -- (1.6, -.4) -- (1.6, .4) -- (2.4, .4) -- (2, .4); 

\end{tikzpicture}
\end{center}
\end{minipage}
\begin{minipage}{1.75in}
\begin{center}
\begin{tikzpicture}[->,>=stealth',shorten >=1pt,shorten <=1pt, auto,node
distance=2cm,thick,every loop/.style={<-,shorten <=1pt}]
\tikzstyle{every state}=[fill=gray!20,draw=none,text=black]

\node[circle,draw=black!100, label=right:$$,inner sep=0pt,minimum size=.175cm] (n1) at (0,0) {{}};
\node[circle,fill=darkgreen!50,draw=black!100, label=left:$$,inner sep=0pt,minimum size=.175cm] (nx) at (1,-.5) {{}};
\node[circle,fill=blue!50,draw=black!100, label=right:$$,inner sep=0pt,minimum size=.175cm] (ny) at (1,-1.5) {{}};
\node[circle,fill=red!50,draw=black!100, label=right:$$,inner sep=0pt,minimum size=.175cm] (nz) at (-1,-1) {{}};
\node[circle,draw=black!100, label=right:$$,inner sep=0pt,minimum size=.175cm] (n0) at (0,-2) {{}};
\path (nx) edge[-] node {{}} (n1);
\path (nx) edge[-] node {{}} (ny);
\path (n1) edge[-] node {{}} (nz);
\path (ny) edge[-] node {{}} (n0);
\path (nz) edge[-] node {{}} (n0);

\path (ny) edge[->,dashed,gray] node {{}} (nz);

\path (nx) edge[->,dashed,gray] node {{}} (n0);

\path (nz) edge[->,dashed,gray] node {{}} (nx);

\path (nx) edge[->,>=stealth,bend left,double,gray] node {{}} (ny);
\path (ny) edge[loop right,>=stealth,double,gray] node {{}} (ny);
\path (nz) edge[loop left,>=stealth,double,gray] node {{}} (nz);
\path (n0) edge[loop left,>=stealth,double,gray] node {{}} (n0);
\path (n1) edge[loop left,>=stealth,double,gray] node {{}} (n1);

\end{tikzpicture}
\end{center}
\end{minipage}
\end{center}
\caption{CA frame and associated necessity lattice.}\label{CAFig}
\end{figure}
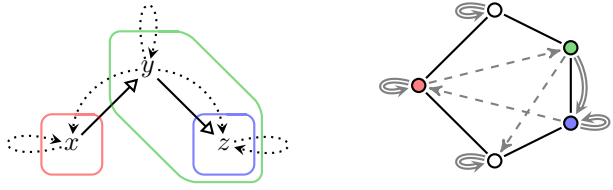

Under certain assumptions about $\comp$, the condition that $\Box_R$ sends $c_\comp$-fixpoints to $c_\comp$-fixpoints corresponds to a first-order condition on $R$ and $\comp$. For example, when $\comp$ is a preorder, in light of Proposition \ref{PWS}(\ref{PWS2}), the condition that $\Box_R$ sends  $c_\comp$-fixpoints to $c_\comp$-fixpoints corresponds to $\Box_R$ sending downsets to downsets and hence to the standard interaction condition for intuitonistic modal frames \cite{Bozic1984}: if $y\in R(x)$ and $x\comp x'$, then there is a $y'\in R(x')$ with $y\comp y'$. When $\comp$ is symmetric, we get the first-order condition in the following proposition. For useful notation, define $z\comp_R x \Leftrightarrow\exists y: z\comp y\in R(x)$.

\begin{proposition}\label{FO} If $(X,\comp)$ is a compatibility frame and $R$ a binary relation on $X$, then $(X,\comp,R)$ is a CA frame if the following condition holds: for all $x,z\in X$, if $z\comp_R x$, then $\exists x'\comp x$ $\forall x''\compflip x' $  \,$z\comp_R x''$. 

Moreover, if $(X,\comp)$ is a symmetric compatibility  frame, then $(X,\comp,R)$ is a CA frame \emph{if and only if} the stated condition holds.\end{proposition}
\noindent In fact, the proof of Proposition \ref{ModalEmbedding} below shows that any necessity lattice can be represented using a CA frame satisfying the condition in Proposition \ref{FO}, so in that sense we can work with such CA frames without loss of generality.

Next we turn to extending the representation theorems of \S~\ref{LatticesToFrames}.

\begin{proposition}\label{SpecialProp2} Let $L$ be a complete lattice satisfying the hypotheses of Proposition \ref{SpecialRep}, so $L$ is isomorphic to $\lat(\mathrm{V},\comp)$ via  $b\mapsto\varphi(b)=\{x\in\mathrm{V}\mid x\leq b\}$. Given a completely multiplicative operation $\Box$ on $L$, define $R$ on $\mathrm{V}$ by $xRy$ iff  $y\leq \bigwedge\{a\in L\mid x\leq\Box a\}$. Then $(\mathrm{V},\comp,R)$ is a CA frame, and $\varphi$ is an isomorphism from $(L,\Box)$ to $(\lat(\mathrm{V},\comp),\Box_R)$.
\end{proposition}

\begin{example} Fig.~\ref{FinalFig} shows a necessity ortholattice $(L,\Box)$ (right) along with the CA frame (left) that comes from the representation of $(L,\Box)$ via join-irreducible elements given by Theorem \ref{RepThm1} and Proposition \ref{SpecialProp2}. It is argued in \cite{HM2021} that this necessity ortholattice captures some important logical entailments involving the epistemic modals `must'  (formalized as $\Box$) and `might' (formalized as $\Diamond=\neg\Box\neg$) in natural language, including the phenomenon of ``epistemic contradiction'' whereby sentences of the form ``$p$, but it might be that $\neg p$'' are judged contradictory, even though ``it might be that $\neg p$'' does not entail $\neg p$.
\end{example}
\begin{figure}[h]

\begin{center}
\begin{minipage}{1.6in}
\begin{center}
\tikzset{every loop/.style={min distance=10mm,looseness=10}}
\begin{tikzpicture}[->,>=stealth,shorten >=1pt,shorten <=1pt, auto,node
distance=2.5cm,semithick]
\tikzstyle{every state}=[fill=gray!20,draw=none,text=black]

\tikzset{every loop/.style={min distance=12mm,looseness=10}}

\node[circle,draw=black!100, fill=black!100, label=below:$$,inner sep=0pt,minimum size=.175cm] (1) at (0,-1.5) {{}};
\node[circle,draw=black!100, fill=black!100, label=below:$$,inner sep=0pt,minimum size=.175cm] (2) at (0,0) {{}};
\node[circle,draw=black!100, fill=black!100, label=below:$$,inner sep=0pt,minimum size=.175cm] (3) at (1.5,0) {{}};
\node[circle,draw=black!100, fill=black!100, label=below:$$,inner sep=0pt,minimum size=.175cm] (4) at (3,0) {{}};
\node[circle,draw=black!100, fill=black!100, label=below:$$,inner sep=0pt,minimum size=.175cm] (5) at (3,-1.5) {{}};

\path (1) edge[{Triangle[open]}-{Triangle[open]}] node {{}} (2);
\path (2) edge[{Triangle[open]}-{Triangle[open]}] node {{}} (3);
\path (3) edge[{Triangle[open]}-{Triangle[open]}] node {{}} (4);
\path (4) edge[{Triangle[open]}-{Triangle[open]}] node {{}} (5);

\path (1) edge[loop left, dotted] node {{}} (1);
\path (2) edge[loop left, dotted] node {{}} (2);
\path (3) edge[loop above, dotted] node {{}} (3);
\path (4) edge[loop right=40, dotted] node {{}} (4);
\path (5) edge[loop right=40, dotted] node {{}} (5);

\path (2) edge[bend left=40, dotted] node {{}} (1);
\path (2) edge[bend right=40, dotted] node {{}} (3);

\path (4) edge[bend left=40, dotted] node {{}} (3);
\path (4) edge[bend right=40, dotted] node {{}} (5);

\path[-, draw=red, opacity=0.5, thick, rounded corners]  (0, -1.9) -- (.4, -1.9) -- (.4, -1.1) -- (-.4, -1.1) -- (-.4, -1.9) -- (.4, -1.9) -- (0, -1.9); 

\path[-, draw=teal, opacity=0.5, thick, rounded corners]  (-.1, -2) -- (.5, -2) -- (.5, .5) -- (-.5, .5) -- (-.5, -2) -- (.5, -2) -- (-.1, -2); 

\path[-, draw=orange, opacity=0.5, thick, rounded corners]  (3, -1.9) -- (3.4, -1.9) -- (3.4, -1.1) -- (2.6, -1.1) -- (2.6, -1.9) -- (3.4, -1.9) -- (3, -1.9); 

\path[-, draw=cyan, opacity=0.5, thick, rounded corners]  (2.9, -2) -- (3.5, -2) -- (3.5, .5) -- (2.5, .5) -- (2.5, -2) -- (3.5, -2) -- (2.9, -2); 

\path[-, draw=brown, opacity=0.5, thick, rounded corners]  (1.5, -.4) -- (1.9, -.4) -- (1.9, .4) -- (1.1, .4) -- (1.1, -.4) -- (1.9, -.4) -- (1.5, -.4); 

\path[-, draw=violet, opacity=0.5, thick, rounded corners]  (-.2, -2.1) -- (3.6, -2.1) -- (3.6, -.9) -- (-.6, -.9) -- (-.6, -2.1) -- (.6, -2.1) -- (-.2, -2.1); 

\path[-, draw=darkgreen, opacity=0.5, thick, rounded corners] (2, .5) -- (2, -.4) --  (.7, -2.2) -- (-.7, -2.2) -- (-.7, .6)  -- (2, .6) -- (2, .5) ; 

\path[-, draw=blue, opacity=0.5, thick, rounded corners]   (1, .5) -- (1, .6) --  (3.7, .6) -- (3.7, -2.2) --  (2.3, -2.2) -- (1, -.4) -- (1, .5)  ; 

\end{tikzpicture}
\end{center}
\end{minipage}\qquad\qquad\;\,\begin{minipage}{2.5in}\begin{center}\begin{tikzpicture}[->,>=stealth',shorten >=1pt,shorten <=1pt, auto,node
distance=2.5cm,semithick]
\tikzstyle{every state}=[fill=gray!20,draw=none,text=black]
\node[circle,draw=black!100, label=below:$$,inner sep=0pt,minimum size=.175cm] (0) at (0,.5) {{}};
\node[circle,draw=black!100, fill=brown!50, label=right:$\textcolor{black}{\Diamond a \wedge \Diamond \neg a}$,inner sep=0pt,minimum size=.175cm] (d) at (0,1.5) {{}};

\node[circle,draw=black!100, fill=red!50, label=below:$\textcolor{black}{\Box a}$,inner sep=0pt,minimum size=.175cm] (a) at (-2.25,1.5) {{}};
\node[circle,draw=black!100, fill=orange!50, label=below:$\;\;\;\;\textcolor{black}{\Box\neg a}$,inner sep=0pt,minimum size=.175cm] (Nc) at (2.25,1.5) {{}};

\node[circle,draw=black!100, fill=teal!50, label=left:$\textcolor{black}{a}$,inner sep=0pt,minimum size=.175cm] (b) at (-2.25,2.5) {{}};
\node[circle,draw=black!100, fill=cyan!50, label=right:$\textcolor{black}{\neg a}$,inner sep=0pt,minimum size=.175cm] (Nb) at (2.25,2.5) {{}};
\node[circle,draw=black!100, fill=darkgreen!50,label=above:$\textcolor{black}{\Diamond a}$,inner sep=0pt,minimum size=.175cm] (c) at (-2.25,3.5) {{}};

\node[circle,draw=black!100, fill=blue!50, label=above:$\;\;\;\textcolor{black}{\Diamond\neg a}$,inner sep=0pt,minimum size=.175cm] (Na) at (2.25,3.5) {{}};

\node[circle,draw=black!100, fill=violet!50, label=right:$\textcolor{black}{\Box a \vee\Box\neg a}\;\;$,inner sep=0pt,minimum size=.175cm] (Nd) at (0,3.5) {{}};
\node[circle,draw=black!100,label=above:$$,inner sep=0pt,minimum size=.175cm] (1) at (0,4.5) {{}};

\path (d) edge[-] node {{}} (c);
\path (d) edge[-] node {{}} (Na);
\path (d) edge[-] node {{}} (0);

\path (1) edge[-] node {{}} (Nd);
\path (Nd) edge[-] node {{}} (a);
\path (Nd) edge[-] node {{}} (Nc);
\path (1) edge[-] node {{}} (c);
\path (1) edge[-] node {{}} (Na);
\path (a) edge[-] node {{}} (b);
\path (b) edge[-] node {{}} (c);
\path (Nc) edge[-] node {{}} (Nb);
\path (Nb) edge[-] node {{}} (Na);
\path (a) edge[-] node {{}} (0);
\path (Nc) edge[-] node {{}} (0);

\path (a) edge[<->, dashed, gray] node {{}} (Na);

\path (b) edge[<->, dashed, gray] node {{}} (Nb);

\path (c) edge[<->, dashed, gray] node {{}} (Nc);

\path (d) edge[<->, dashed, gray] node {{}} (Nd);

\path (Na) edge[loop right,gray,double,>=stealth] node {{}} (Na);
\path (d) edge[loop left,gray,double,>=stealth] node {{}} (d);
\path (Nd) edge[loop left,gray,double,>=stealth] node {{}} (Nd);
\path (1) edge[loop left,gray,double,>=stealth] node {{}} (1);
\path (c) edge[loop left,gray,double,>=stealth] node {{}} (c);
\path (a) edge[loop left,gray,double,>=stealth] node {{}} (a);
\path (Nc) edge[loop right,gray,double,>=stealth] node {{}} (Nc);
\path (0) edge[loop left,gray,double,>=stealth] node {{}} (0);

\path (b) edge[->,gray,bend right,double,>=stealth] node {{}} (a);
\path (Nb) edge[->,gray,bend left,double,>=stealth] node {{}} (Nc);
\end{tikzpicture}
\end{center}
\end{minipage}
\end{center}

\caption{CA frame realizing a necessity ortholattice (with $\Diamond$ defined by $\neg\Box\neg$).}\label{FinalFig}
\end{figure}
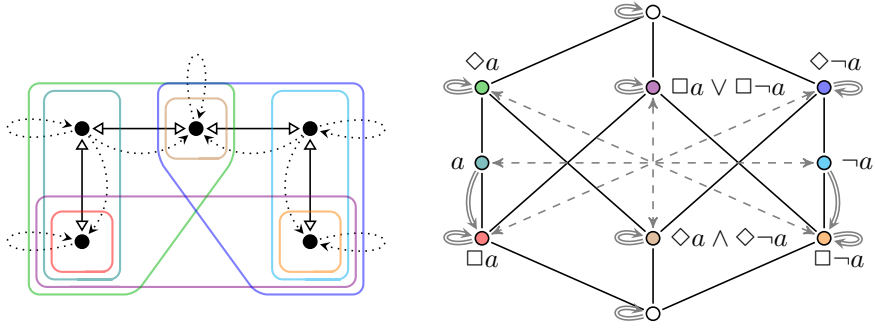

 We can extend our other representation results to the modal setting as well. Given $(L,\Box)$ and $P$ a separating set of pairs of elements of $L$ as in \S~\ref{ArbCompLat}, define a relation $R$ on $P$ by $(x,x')R(y,y')$ iff $xRy$ as defined in Proposition \ref{SpecialProp2}.

\begin{proposition}\label{ModalCompRep} If $(L,\Box)$ is a complete necessity lattice with $\Box$ completely multiplicative and $P$ a separating set of pairs of elements of $L$, then $(P,\comp,R)$ is a CA frame and  $(L,\Box)$ is isomorphic to $(\lat(P,\comp), \Box_R)$.
\end{proposition}

Using Theorem \ref{NegThm} we can similarly represent complete necessity lattices equipped with a protocomplementation.

\begin{proposition}\label{BoxNeg} If $(L,\Box)$ is a complete necessity lattice with $\Box$ completely multiplicative and $\neg$ is a protocomplementation on $L$, then there is a CA frame $(P,\comp,R)$ such that $(L,\Box,\neg)$ is isomorphic to $(\lat(P,\comp), \Box_R, \neg_\comp)$.
\end{proposition} 

Finally, define $\mathsf{FI}(L,\Box)$ just like $\mathsf{FI}(L)$ in \S~\ref{SubRepSec} but with the addition of a relation $R$ with $(F,I)R(F',I')$ if for all $a\in L$,  $\Box a\in F$ implies $a\in F'$.

\begin{proposition}\label{ModalEmbedding} For any necessity lattice $(L,\Box)$, $\mathsf{FI}(L,\Box)$ is a CA frame, and the map $a\mapsto\widehat{a}$ is \textnormal{(i)} an embedding of $(L,\Box)$ into $(\lat(\mathsf{FI}(L)),\Box_R)$ and \textnormal{(ii)} an isomorphism from $(L,\Box)$ to the subalgebra of $(\lat(\mathsf{FI}(L)),\Box_R)$ consisting of $c_\comp$-fixpoints that are compact open in the space $\mathsf{S}(L)$ \textnormal{(recall \S~\ref{SubRepSec})}.\end{proposition}

\section{Conclusion}\label{Conclusion}

We have investigated three representations of complete lattices by means of compatibility frames, as well as modal analogues thereof. For future work, we hope to make progress on Question \ref{CharQues} and Conjecture \ref{Conj}, as well as applications of the representations studied here to lattice-based logics. For modal logic in particular, for reasons in \cite{HM2021} we would like to understand the lattice of modal orthologics, for which we hope that CA frames will be useful.

\subsection*{Acknowledgements}

I thank Peter Jipsen, Guillaume Massas, and the referees for \textit{Advances in Modal Logic} for helpful feedback.

\appendix 

\section{Appendix}

\subsection{Representation of $(L,\neg)$}\label{NegAppendix}

In Theorem \ref{RepThm1}(\ref{RepThm1-3}), the conclusion is that  $L$ is isomorphic to $\lat(\mathrm{V},\comp^\neg_\mathrm{V})$, not that $(L,\neg)$ is isomorphic to  $(\lat(\mathrm{V},\comp^\neg_\mathrm{V}), \neg_{\comp^\neg_\mathrm{V}})$. To represent $(L,\neg)$, we ask for a \textit{third suitability condition} on $\comp^\neg_\mathrm{V}$, namely that for $x\in\mathrm{V}$ and $y\in L$:
\begin{itemize}
\item[] if $x\not\leq \neg y$, then there is a $y_0\in\mathrm{V}$ such that $y_0\leq y$ and $y_0\comp^\neg_\mathrm{V}x$.
\end{itemize}

\begin{proposition}\label{EconRepNeg}  If $L$ is a complete lattice, $\mathrm{V}$ is a join-dense set of nonzero elements, $\neg$ is a protocomplementation on $L$ such that $L$ has compatible escape in $\mathrm{V}$ with $\neg$, and $\comp^\neg_\mathrm{V}$ satisfies the second and third suitability conditions, then $(L,\neg)$ is isomorphic to $(\lat(\mathrm{V},\comp^\neg_\mathrm{V}),\neg_{\comp^\neg_\mathrm{V}})$.
\end{proposition}

\begin{proof} We need only add to the proof of Theorem \ref{RepThm1}(\ref{RepThm1-3}) that $\varphi(\neg a)=\neg_{\comp^\neg_\mathrm{V}} \varphi(a)$.  Suppose $x\in\varphi(\neg a)$, so $x\leq \neg a$, and $x'\comp^\neg_\mathrm{V} x$, so $x\not\leq \neg x'$. If $x'\leq a$, then $\neg a\leq\neg x'$, which with $x\leq \neg a$ implies $x\leq \neg x'$, contradicting the previous sentence. Thus, $x'\not\leq a$, so $x'\not\in\varphi(a)$. Hence $x\in \neg_{\comp^\neg_\mathrm{V}} \varphi(a)$. Conversely, suppose $x\in\mathrm{V}\setminus\varphi(\neg a)$, so $x\not\leq \neg a$. Then by the third suitability condition, there is an $a_0\in\mathrm{V}$ such that $a_0\leq a$, so $a_0\in\varphi(a)$, and $a_0\comp^\neg_\mathrm{V} x$, so $x\not\in \neg_{\comp^\neg_\mathrm{V}} \varphi(a)$.\end{proof}

\noindent It remains to be seen how broadly Proposition \ref{EconRepNeg} applies to lattices expanded with a protocomplementation. In the Jupyter notebook cited in \S~\ref{Intro}, we show there are such expansions $(L,\neg)$ that cannot be represented by any compatibility frame with $|X|\leq |L|$. Of course, Proposition \ref{EconRepNeg} applies to all ortholattices.

\begin{proposition}\label{EconRepNegOrtho} If $(L,\neg)$ is a complete ortholattice, and $\mathrm{V}$ is a join-dense set of nonzero elements, then $\comp^\neg_\mathrm{V}$ satisfies the third suitability condition.
\end{proposition}

\begin{proof} If $x\not\leq \neg y$, then $x\not\leq \neg \bigvee \{z\in \mathrm{V}\mid z\leq y\} =\bigwedge \{\neg z\mid z\in \mathrm{V},\, z\leq y\}$, so there is $y_0\in \mathrm{V}$ with $y_0\leq y$ and $x\not\leq \neg y_0$, so $y_0\comp^\neg_\mathrm{V}x$ by Proposition \ref{KeyDefApplies}(\ref{KeyDefApplies1}).\end{proof}

Recall that using the less economical representation of \S~\ref{ArbCompLat},  any complete lattice expanded with a protocomplementation is representable (Theorem \ref{NegThm}).

\subsection{Proofs for \S~\ref{SubRepSec}}\label{SubAppendix}

\begin{manualtheorem}{\ref{EmbedThm}} For any lattice \textnormal{(}resp.~bounded lattice\textnormal{)} $L$, the map $a\mapsto\widehat{a}$ is \textnormal{(i)} a lattice \textnormal{(}resp.~bounded lattice\textnormal{)} embedding  of $L$ into $\lat(\mathsf{FI}(L))$ and \textnormal{(ii)} an isomorphism from $L$ to the sublattice of $\lat(\mathsf{FI}(L))$ consisting of $c_\comp$-fixpoints that are compact open in the space $\mathsf{S}(L)$.\end{manualtheorem}

\begin{proof} First observe that for any $a\in L$,  $\widehat{a}$ is a $c_\comp$-fixpoint. It suffices to show that if $(F,I)\not\in \widehat{a}$, then there is an $(F',I')\comp (F,I)$ such that for all $(F'',I'')\compflip (F',F')$, we have $(F'',I'')\not\in \widehat{a}$. Suppose $(F,I)\not\in\widehat{a}$, so $a\not\in F$. Let $F'=F$ and $I'=\mathord{\downarrow}a$. Then $a\not\in F$ implies $F'\cap I'=\varnothing$. Thus, $(F',I')\in \mathsf{FI}(L)$. Now consider any $(F'',I'')$ such that $(F',I')\comp (F'',I'')$, so $I'\cap F''=\varnothing$. Then since $a\in I'$, we have $a\not\in F''$, so $(F'',I'')\not\in \widehat{a}$, as desired.

Next, the map $a\mapsto\widehat{a}$ is clearly injective:  if $a\not\leq b$, then $\mathord{\uparrow}a\cap \mathord{\downarrow}b=\varnothing$, so $(\mathord{\uparrow}a, \mathord{\downarrow}b)\in \mathsf{FI}(L)$, $(\mathord{\uparrow}a, \mathord{\downarrow}b)\in \widehat{a}$, and $(\mathord{\uparrow}a, \mathord{\downarrow}b)\not\in\widehat{b}$. If $L$ has bounds, then  $\widehat{1}=X$ and $\widehat{0}=\varnothing$. The map also preserves $\wedge$:  $\widehat{a\wedge b}=\{(F,I)\in X\mid a\wedge b\in F\}=\{(F,I)\in X\mid a,b\in F\}=\{(F,I)\in X\mid a\in F\}\cap \{(F,I)\in X\mid b\in F\}=\widehat{a}\cap \widehat{b}=\widehat{a}\wedge\widehat{b}$. 

To complete part (i), we show $\widehat{a\vee b}\subseteq\widehat{a}\vee\widehat{b}$, as the converse inclusion follows from meet preservation. Recall from Proposition \ref{ClosureLattice} that ${\widehat{a}\vee\widehat{b}}=c_\comp(\widehat{a}\cup\widehat{b})$. Suppose $(F,I)\in \widehat{a\vee b}$, so $a\vee b\in F$. Consider any ${(F',I')\comp (F,I)}$, so $I'\cap F=\varnothing$ and hence $a\vee b\not\in I'$. Then since $I'$ is an ideal,  $a\not\in I'$ or ${b\not\in I'}$. Without loss of generality, suppose $a\not \in I'$. Then setting $F''=\mathord{\uparrow}a$ and $I''=\mathord{\downarrow}c$ for any $c\in I'$, we have $(F'',I'')\in \mathsf{FI}(L)$ and $I'\cap F''=\varnothing$, so ${(F',I')\comp (F'',I'')}$, and $(F'',I'')\in\widehat{a}$. Thus, we have shown that for any $(F',I')\comp (F,I)$ there is an $(F'',I'')\compflip (F',I')$ with $(F'',I'')\in\widehat{a}\cup \widehat{b}$. Hence $(F,I)\in\widehat{a}\vee\widehat{b}$. 

For part (ii), we first show that $\widehat{a}$ is compact open. Since $\widehat{b}$'s form a basis, we need only show that if $\widehat{a}\subseteq \bigcup \{\widehat{b}_i\mid i\in I\}$, then there is a finite subcover. Suppose $a\not\leq b_i$ for some $i\in I$.  Then since $(\mathord{\uparrow}a,\mathord{\downarrow}b_i)\in \widehat{a}$, we have $(\mathord{\uparrow}a,\mathord{\downarrow}b_i)\in \widehat{b_j}$ for some $j\in I$, which implies $a\leq b_j$. Thus, $a\leq b_k$ for some $k\in I$, so $\widehat{a}\subseteq\widehat{b_k}$.

Finally, we show that $a\mapsto\widehat{a}$  is onto the set of compact open $c_\comp$-fixpoints. Suppose $U$ is compact open, so $U=\widehat{a_1}\cup\dots\cup\widehat{a_n}$ for some $a_1,\dots,a_n\in L$. Further suppose $U$ is a $c_\comp$-fixpoint, so $c_\comp(U)=U$. Where $d=a_1\vee\dots\vee a_n$, an obvious induction using part (i) and the fact that $c_\comp(c_\comp(A)\cup B)=c_\comp(A\cup B)$ for any $A,B\subseteq X$ yields $\widehat{d}= c_\comp(\widehat{a_1}\cup\dots\cup\widehat{a_n})$, so $\widehat{d}=c_\comp (U)= U$.\end{proof}

\begin{manualprop}{\ref{COFixProp}} For any space $X$ and reflexive binary relation $\comp$ on $X$, there is a lattice $L$ such that $(X,\comp)$ and $\mathsf{S}(L)$ are homeomorphic as spaces and isomorphic as relational frames iff the following conditions hold for all $x,y\in X$: \textnormal{(i)} $x=y$ iff $(\mathsf{F}(x),\mathsf{I}(x))=(\mathsf{F}(y),\mathsf{I}(y))$\textnormal{;} \textnormal{(ii)} $\mathsf{COFix}(X,\comp)$ is closed under $\cap$ and $\vee$ and forms a basis for $X$\textnormal{;} \textnormal{(iii)} each disjoint filter-ideal pair from  $\mathsf{COFix}(X,\comp)$ is $(\mathsf{F}(x), \mathsf{I}(x))$ for some $x\in X$\textnormal{;} \textnormal{(iv)} $x\comp y$ iff $\mathsf{I}(x)\cap \mathsf{F}(y)= \varnothing$.
\end{manualprop}
\begin{proof} Suppose there is such an $L$. It suffices to show $\mathsf{S}(L)$ satisfies (i)--(iv) in place of $(X,\comp)$. That (ii) holds for  $\mathsf{COFix}(\mathsf{S}(L))$ and $\mathsf{S}(L)$ follows from the proof of Theorem \ref{EmbedThm}. Let $\varphi$ be the isomorphism $a\mapsto\widehat{a}$ from $L$ to $\mathsf{COFix}(\mathsf{S}(L))$ in Theorem \ref{EmbedThm}, which induces a bijection $(F,I)\mapsto (\varphi[F],\varphi[I])$ between disjoint filter-ideal pairs of $L$ and of $\mathsf{COFix}(\mathsf{S}(L))$. Parts (i), (iii), and (iv) follow from the fact that for any $x=(F,I)\in \mathsf{S}(L)$, $(\varphi[F],\varphi[I])=(\mathsf{F}(x),\mathsf{I}(x))$. First, $\widehat{a}\in \varphi[F]$ iff $a\in F$ iff $x\in \widehat{a}$ iff $\widehat{a}\in \mathsf{F}(x)$. Second, $\widehat{a}\in \varphi[I]$ iff $a\in I$, and we claim that $a\in I$ iff $\widehat{a}\in \mathsf{I}(x)$, i.e., for all $(F',I')\compflip (F,I)$, $(F',I')\not\in \widehat{a}$, i.e., $a\not\in F'$. If $a\in I$ and $(F,I)\comp (F',I')$, then $a\not\in F'$ by definition of $\comp$. Conversely, if $a\not\in I$, let $F'=\mathord{\uparrow}a$ and $I'=I$; then $(F,I)\comp (F',I')$ and $a\in F'$. Now for (i), given $x,y\in \mathsf{S}(L)$ with $x=(F,I)$ and $y=(F',I')$, we have $(F,I)=(F',I')$ iff $(\varphi[F], \varphi[I]) = (\varphi[F'], \varphi[I'])$ iff  $(\mathsf{F}(x),\mathsf{I}(x))=(\mathsf{F}(y),\mathsf{I}(y))$; similarly, for (iv), $(F,I)\comp (F',I')$ iff $I\cap F'=\varnothing$ iff $\varphi[I]\cap \varphi[F']=\varnothing$ iff $\mathsf{I}(x)\cap \mathsf{F}(y)=\varnothing$. Finally, for (iii), if $(\mathscr{F},\mathscr{I})$ if a disjoint filter-ideal pair from $\mathsf{COFix}(\mathsf{S}(L))$, then setting $x=(\varphi^{-1}[\mathscr{F}],\varphi^{-1}[\mathscr{I}])$, we have $x\in \mathsf{S}(L)$ and $(\mathscr{F},\mathscr{I}) =(\mathsf{F}(x), \mathsf{I}(x))$.

Assuming $X$ satisfies the conditions, $\mathsf{COFix}(X,\comp)$ is a lattice, and we define a map $\epsilon$ from $(X,\comp)$ to $\mathsf{S}(\mathsf{COFix}(X,\comp))$  by $\epsilon(x)= (\mathsf{F}(x),\mathsf{I}(x))$. The proof that $\epsilon$ is a homeomorphism using (i)--(iii) is analogous to the proof of  Thm.~5.4(2) in \cite{BH2020}. That $\epsilon$ preserves and reflects $\comp$ follows from (iv).
\end{proof}

\subsection{Proofs for \S~\ref{ModalCase}}\label{NecAppendix}

\begin{manualprop}{\ref{FO}} If $(X,\comp)$ is a compatibility frame and $R$ a binary relation on $X$, then $(X,\comp,R)$ is a CA frame if the following condition holds: for all $x,z\in X$, if $z\comp_R x$, then $\exists x'\comp x$ $\forall x''\compflip x' $  \,$z\comp_R x''$. 

Moreover, if $(X,\comp)$ is a symmetric compatibility  frame, then $(X,\comp,R)$ is a CA frame \emph{if and only if} the stated condition holds.\end{manualprop}

\begin{proof}For the first part, we must show that $\Box_R A$ is a $c_\comp$-fixpoint for any $c_\comp$-fixpoint $A$. That is, we must show that
$ x\in X\setminus \Box_R A \Rightarrow  \exists x'\comp x\,\forall x''\compflip x' \; x''\not\in \Box_R A$.  Suppose $x\not\in \Box_R A$, so there is some $y\in R(x)$ with $y\not\in A$. Then since $A$ is a {$c_\comp$-fixpoint}, there is a ${z\comp y}$ such that ($\star$) for all $z'\compflip z$, we have $z'\not\in A$. Since $z\comp y\in R(x)$, by the condition we have $\exists x'\comp x\,\forall x''\compflip x' \,\exists y':\, z\comp y' \in R(x'')$. Now $z\comp y'$ implies $y'\not\in A$ by ($\star$), which with $y'\in R(x'')$ implies $x''\not\in\Box_R A$.

For the second part, assume $\comp$ is symmetric and $\Box_R$ sends $c_\comp$-fixpoints to $c_\comp$-fixpoints. Toward proving the condition, suppose $z\comp y\in R(x)$. Hence ${y\not\in \neg_\comp c_\comp(\{z\})}$, so  $x\not\in \Box_R \neg_\comp c_\comp(\{z\})$. Since by assumption $\Box_R \neg_\comp c_\comp(\{z\})$ is a $c_\comp$-fixpoint, it follows that there is an $x'\comp x$ such that for all $x''\compflip x'$, we have that $x''\not\in\Box_R\neg_\comp c_\comp(\{z\})$. Thus, there is a $w\in R(x'')$ such that $w\not\in \neg_\comp c_\comp(\{z\})$, so for some $w'\comp w$, we have $w'\in c_\comp(\{z\})$, which means that for all $w''\comp w'$, there is a $w'''\compflip w''$ such that $w'''\in\{z\}$, i.e., for all $w''\comp w'$, $w''\comp z$. Since $\comp$ is symmetric, from $w'\comp w$, we have $w\comp w'$, so setting $w''=w$, we conclude $w\comp z$, so $z\comp w$. Thus, $z\comp w\in R(x'')$, i.e., $z\comp_R x''$. 
\end{proof}

\begin{manualprop}{\ref{SpecialProp2}} Let $L$ be a complete lattice satisfying the hypotheses of Proposition \ref{SpecialRep}, so $L$ is isomorphic to $\lat(\mathrm{V},\comp)$ via  $b\mapsto\varphi(b)=\{x\in\mathrm{V}\mid x\leq b\}$. Given a completely multiplicative operation $\Box$ on $L$, define $R$ on $\mathrm{V}$ by $xRy$ iff  $y\leq \bigwedge\{a\in L\mid x\leq\Box a\}$. Then $(\mathrm{V},\comp,R)$ is a CA frame, and $\varphi$ is an isomorphism from $(L,\Box)$ to $(\lat(\mathrm{V},\comp),\Box_R)$.\end{manualprop}

\begin{proof} First, recall the key fact provided by complete multiplicativity of $\Box$: if $x\not\leq \Box b$, then $\bigwedge\{a\in L\mid x\leq\Box a\}\not\leq b$. For if $\bigwedge\{a\in L\mid x\leq\Box a\}\leq b$, then $\Box\bigwedge\{a\in L\mid x\leq\Box a\}\leq \Box b$ and hence $\bigwedge\{\Box a\in L\mid x\leq\Box a\}\leq \Box b$, so $x\leq\Box b$.

Now we show that for all $b\in L$, $\varphi(\Box b)=\Box_R\varphi(b)$. Suppose $x\in \varphi(\Box b)$, so $x\leq\Box b$. Then for all $y\in R(x)$, we have $y\leq b$ and hence $y\in \varphi(b)$. Thus, $x\in \Box_R\varphi(b)$. Now suppose $x\not\in\varphi(\Box b)$, so $x\not\leq\Box b$. Hence $\bigwedge\{a\in L\mid x\leq\Box a\}\not\leq b$ as above. Then since $\mathrm{V}$ is join-dense, there is a $y\in\mathrm{V}$ such that $y\leq \bigwedge\{a\in L\mid x\leq\Box a\}$ but $y\not\leq b$. Hence $xRy$ and $y\not\in\varphi(b)$, so $x\not\in \Box_R\varphi(b)$. 

Finally, we prove that $(\mathrm{V},\comp,R)$ is indeed a CA frame: if $B$ is a $c_{\comp}$-fixpoint of $(\mathrm{V},\comp)$, so is $\Box_R B$. By the surjectivity of $\varphi$, $B=\varphi(b)$ for some $b\in B$. Then $\Box_RB= \Box_R\varphi(b)=\varphi(\Box b)$, and $\varphi(\Box b)$ is a $c_{\comp}$-fixpoint, so we are done.\end{proof}

\begin{manualprop}{\ref{ModalCompRep}} If $(L,\Box)$ is a complete necessity lattice with $\Box$ completely multiplicative and $P$ a separating set of pairs of elements of $L$, then $(P,\comp,R)$ is a CA frame and  $(L,\Box)$ is isomorphic to $(\lat(P,\comp), \Box_R)$.
\end{manualprop}

\begin{proof}We showed in the proof of Proposition \ref{CompRep} that $a\mapsto\varphi(a)=\{(x,y)\in P\mid x\leq a\}$ is an isomorphism from $L$ to $\lat(P,\comp)$. It only remains to show that $\varphi(\Box b)=\Box_R\varphi(b)$. Suppose $(x,y)\in \varphi(\Box b)$, so $x\leq \Box b$. Then $(x,y)R(x',y')$ implies $x'\leq b$ and hence $(x',y')\in \varphi(b)$. Thus, $(x,y)\in\Box_R\varphi(b)$. Conversely, suppose $(x,y)\not\in \varphi(\Box b)$, so $x\not\leq \Box b$. Since $\Box$ is completely multiplicative, it follows that $\bigwedge \{a\in L\mid x\leq \Box a\}\not\leq b$. Then since $P$ is separating, there is some $(c,d)\in P$ with $c\leq \bigwedge \{a\in L\mid x\leq \Box a\}$ but $c\not\leq b$. Hence $(x,y)R(c,d)$ but $c\not\in \varphi(b)$, so $(x,y)\not\in \Box_R\varphi(b)$. Now the proof that $(P,\comp,R)$ is a CA frame is analogous to the last paragraph of the previous proof.\end{proof}

\begin{manualprop}{\ref{BoxNeg}} If $(L,\Box)$ is a complete necessity lattice with $\Box$ completely multiplicative and $\neg$ is a protocomplementation on $L$, then there is a CA frame $(P,\comp,R)$ such that $(L,\Box,\neg)$ is isomorphic to $(\lat(P,\comp), \Box_R, \neg_\comp)$.
\end{manualprop}
\begin{proof} Where $P=\{(a,b)\mid a,b\in L, a\neq 0,a\not\leq b,\neg a\leq b\}$, we showed in the proof of Theorem \ref{NegThm} that $P$ is separating. Hence by the proof of Proposition \ref{ModalCompRep}, $(L,\Box)$ is isomorphic to $(\lat(P,\comp), \Box_R)$ via the map $a\mapsto\varphi(a)=\{(x,y)\in P\mid x\leq a\}$. We also showed in the proof of Theorem \ref{NegThm} that $\varphi$ preserves the protocomplementation. Hence $(L,\Box,\neg)$ is isomorphic to $(\lat(P,\comp), \Box_R, \neg_\comp)$.\end{proof}

\begin{manualprop}{\ref{ModalEmbedding}} For any necessity lattice $(L,\Box)$, $\mathsf{FI}(L,\Box)$ is a CA frame, and the map $a\mapsto\widehat{a}$ is \textnormal{(i)} an embedding of $(L,\Box)$ into $(\lat(\mathsf{FI}(L)),\Box_R)$ and \textnormal{(ii)} an isomorphism from $(L,\Box)$ to the subalgebra of $(\lat(\mathsf{FI}(L)),\Box_R)$ consisting of $c_\comp$-fixpoints that are compact open in the space $\mathsf{S}(L)$.
\end{manualprop}

\begin{proof}In the proof of Proposition \ref{FO}, we showed that for any $(X,\comp)$ and binary relation $R$ on $X$, if the first-order condition in Proposition \ref{FO} holds, then $(X,\comp,R)$ is a CA frame. We claim $\mathsf{FI}(L,\Box)$ satisfies the condition. Suppose $(G,H)\comp (G',H')\in R((F,I))$, which implies $H\cap \{a\in L\mid \Box a\in F\}=\varnothing$. Then where $F'=F$ and $I'$ is the ideal generated by $\{\Box a\mid a\in H\}$, we claim $F'\cap I'=\varnothing$, so $(F',I')\comp (F,I)$.  For if $b\in F'\cap I'$, then for some  $a_1,\dots,a_n\in H$, $b\leq \Box a_1\vee\dots\vee \Box a_n$, which implies $b\leq \Box(a_1\vee\dots\vee a_n)$, so $\Box(a_1\vee\dots\vee a_n)\in F$, whence $a_1\vee\dots\vee a_n\not\in H$, contradicting $a_1,\dots,a_n\in H$. Now suppose $(F',I')\comp (F'',I'')$, so $I'\cap F''=\varnothing$. Let $J=\{b\in L\mid \Box b\in F''\}$, which is a filter. We claim $J\cap H=\varnothing$. For if $a\in J$, then  $\Box a\in F''$, so $\Box a\not\in I'$, whence $a\not\in H$. Thus, $(G,H)\comp (J,H)\in R((F'',I''))$, as desired.

Now for (i)--(ii), we need only add to Theorem \ref{EmbedThm} that $\widehat{\Box a}=\Box_R\widehat{a}$. Suppose $(F,I)\in \widehat{\Box a}$, so $\Box a\in F$. Then if $(F,I)R(F',I')$, we have $a\in F'$ and hence $(F',I')\in\widehat{a}$. Thus, $(F,I)\in\Box_R\widehat{a}$. Conversely, suppose $(F,I)\not\in \widehat{\Box a}$, so $\Box a\not\in F$. Let $F'$ be the filter generated by $\{b\in L\mid \Box b\in F\}$ and $I'=\mathord{\downarrow}a$. Since $\Box a\not\in F$ and $\Box$ is multiplicative, it follows that $F'\cap I'=\varnothing$. Hence $(F',I')\in \mathsf{FI}(L,\Box)$, ${(F',I')\not\in\widehat{a}}$, and  by construction of $F'$, $(F,I)R(F',I')$. Thus $(F,I)\not\in\Box_R\widehat{a}$.\end{proof}

\newpage

\bibliographystyle{aiml22}
\bibliography{aiml22}

\end{document}